\documentclass[11pt]{amsart}
\usepackage{amsmath,amsthm,amsfonts,amssymb,latexsym}
\usepackage[all,2cell,dvips]{xy}
\usepackage{srcltx}
\usepackage{color}

\newtheorem{theorem}{Theorem}[section]
\newtheorem{lemma}[theorem]{Lemma}
\newtheorem{proposition}[theorem]{Proposition}
\newtheorem{corollary}[theorem]{Corollary}

\theoremstyle{definition}
\newtheorem{definition}[theorem]{Definition}

\theoremstyle{remark}
\newtheorem*{remark}{Remark}
\newtheorem*{remarks}{Remarks}

\newcommand{\N}{\mathbb{N}}
\newcommand{\Z}{\mathbb{Z}}
\newcommand{\Q}{\mathbb{Q}}
\newcommand{\R}{\mathbb{R}}
\newcommand{\C}{\mathbb{C}}

\newcommand{\dcl}{\operatorname{dcl}}

\newcommand{\rc}{\operatorname{rc}}

\newcommand{\RCF}{\operatorname{RCF}}

\def \lcm {\operatorname{lcm}}

\newcommand{\la}{\mathcal{L}}
\newcommand{\lam}{\mathcal{L}_{\operatorname{o}}}

\newcommand{\Cal}{\mathcal}

\def \<{\langle}
\def \>{\rangle}

\def \((  {(\!(}
\def \)) {)\!)}

\def \gr{\operatorname{gr}}

\numberwithin{equation}{section}

\allowdisplaybreaks[2]

\makeatother

\def\Ind#1#2{#1\setbox0=\hbox{$#1x$}\kern\wd0\hbox to 0pt{\hss$#1\mid$\hss}
\lower.9\ht0\hbox to 0pt{\hss$#1\smile$\hss}\kern\wd0}

\def\Notind#1#2{#1\setbox0=\hbox{$#1x$}\kern\wd0\hbox to 0pt{\mathchardef
\nn=12854\hss$#1\nn$\kern1.4\wd0\hss}\hbox to
0pt{\hss$#1\mid$\hss}\lower.9\ht0 \hbox to
0pt{\hss$#1\smile$\hss}\kern\wd0}

\begin{document}
%\normalem
\bibliographystyle{plain}
\author{Ayhan G\"{u}nayd{\i}n and Philipp Hieronymi}
\title{The real field with the rational points of an elliptic curve}
\thanks{The second author was funded by Deutscher Akademischer Austausch Dienst.}

%\address{Fields Institute, 222 College Street, Second Floor, Toronto, ON, M5T 3J1, Canada}
\address{Centro de Matem\'{a}tica e Aplica\c{c}\~{o}es Fundamentais,
Av. Prof. Gama Pinto, 2,
1649-003 Lisboa, Portugal}

\address{Department of Mathematics \& Statistics,
McMaster University, 1280 Main Street West
Hamilton, ON, L8S 4K1, Canada}

\email{ayhan@ptmat.fc.ul.pt}

\email{P@hieronymi.de}

\date{\today}

\begin{abstract}
 We consider the expansion of the real field by the group of rational points of an elliptic curve over the rational numbers. We prove a completeness result, followed by a quantifier elimination result. Moreover we show that open sets definable in that structure are semialgebraic.
\end{abstract}

\maketitle

\section{Introduction}

\noindent
Here we study the expansion of the real field by the set, $\Cal C$, of pairs $(x,y)\in\Q^2$ such that
$$y^2=x^3+ax+b,$$
with $a,b\in\Q$ such that $4a^3+27b^2\neq 0$.

\medskip\noindent
We consider $\big(\R,\Cal C\big)$ as a structure in the language $\lam(P)$ extending the language $\lam=\{0,1,+,\cdot,<\}$ of ordered rings by a binary relation symbol $P$. Our main result is the following.
\begin{theorem}\label{nearMC}
Every subset of $\R^s$ definable in the structure $(\R,\Cal C)$ is defined by a boolean combination of formulas of the form
$$\exists y_1\cdots\exists y_{2n}\big[\bigwedge_{j=1}^{n}P(y_{2j-1},y_{2j})\wedge\phi(x,y)\big],$$
where $y$ denotes the tuple $(y_1,\dots,y_{2n})$, $x$ is an $s$-tuple of distinct variables and $\phi(x,y)$ is a quantifier-free $\lam$-formula.
\end{theorem}

\noindent
As a by-product of our techniques, we also axiomatize the first order theory of $(\R,\Cal C)$ (see Theorem \ref{completeness-elliptic}).

\medskip\noindent
One of our motivations for studying $(\R,\Cal C)$ is to understand the open definable sets in the sense of \cite{opencore}.  In the last section we prove the following.

\begin{theorem}\label{opencore-elliptic}
 Let $U\subseteq\R^s$ be an open set definable in $(\R,\Cal C)$. Then $U$ is semialgebraic.
\end{theorem}

\medskip\noindent
%The paper is organized as follows.
We prove Theorem~\ref{nearMC} and Theorem~\ref{opencore-elliptic} for a broader class of structures than the ones in the statements. Namely we study $(\R,\Gamma)$, where $\Gamma\subseteq \R^m$ is a dense subgroup of a one dimensional connected group definable in $\R$, satisfying a number theoretic property. The details of the setting can be found in Section \ref{ML-property}. In Section \ref{modeltheory}, we show that the conclusion of Theorem \ref{nearMC} holds for these structures.
The reader would notice that $\Cal C$, considered as a subset of the projective plane $\mathbb P^2(\C)$, becomes the group of rational points of an elliptic curve after adding a point at infinity. We explain this thoroughly in Section \ref{ellipticcurve} and using this we illustrate how the structure $(\R,\mathcal C)$ fits into the more general framework.

\medskip\noindent
The current paper is not the first attempt to treat such structures. For instance, Zilber studied the real field expanded by the group of roots of unity in \cite{zilber2} and later
Belegradek and Zilber generalized the results of that paper to the real field expanded by a subgroup of the unit circle, of finite rank in \cite{BZ}. The first author of the current paper studied similar structures in \cite{thesis_ayhan} with an approach different than the one in \cite{BZ}. However neither \cite{BZ} nor \cite{thesis_ayhan} prove anything about the structure of open definable sets. Since we prove our theorems in the generality of Section \ref{ML-property}, we were able to get some results in that direction: Let $$\mathbb S:=\{(x,y)\in\R^2:x^2+y^2=1\}$$ be the unit circle and let $\Gamma$ be a finite rank subgroup of $\mathbb S$; that is  $\Gamma$ is contained in the divisible closure of a finitely generated subgroup of $\mathbb{S}$. Now the statement analogous to Theorem~\ref{opencore-elliptic} in this setting is as follows.

\begin{theorem}\label{opencore-circle}  Let $U\subseteq\R^s$ be an open set definable in $(\R,\Gamma)$. Then $U$ is semialgebraic.
\end{theorem}

\bigskip\noindent
{\it Conventions and notations. }Above and in the rest of the paper $m,n,s,t$ always denote natural numbers. Also as usual `definable' means `definable with parameters' and when we want to make the language and the parameters explicit we write $\Cal L$-$B$-{\it definable} to mean definable in the appropriate $\Cal L$-structure using parameters from the set $B$.

\smallskip\noindent
The real closure of an ordered field $K$ is denoted by $K^{\rc}$.

\smallskip\noindent
We denote the graph of a function $f:X\to Y$ by
$\gr(f)$.

\medskip\noindent
{\it Acknowledgements. } We thank L. van den Dries, C. Ealy, C. Miller, J. Ramakrishnan, S. Starchenko and B. Zilber for useful discussions on the topic. We also thank The Fields Institute for hospitality during the `Thematic Program on O-minimal Structures and Real Analytic Geometry'; most of this paper was written during that period.

\section{The Mordell-Lang property}\label{ML-property}
\noindent
Throughout this section $K$ denotes a real closed field and $(\mathbb A,\oplus)$ is a one-dimensional group definable in $K$; that is $\mathbb A\subseteq K^m$ and $\oplus:K^m\times K^m\to K^m$ are definable in $K$ such that $\mathbb A$ is of dimension one and $\oplus|(\mathbb A\times\mathbb A)$ is a group operation on $\mathbb A$. (Here and below we do not make any distinction between an ordered field and its underlying set.)

\medskip\noindent
By Proposition 2.5 of \cite{Pillay}, there is a topology on $\mathbb A$ definable in $K$ such that $\mathbb A$ becomes a topological group. We will refer to this topology as the \emph{t-topology}. Further, a subset of $\mathbb A$ is called \emph{t-dense} (respectively {\it t-connected, t-compact}) if it is dense (respectively connected, compact) in the t-topology. A function $f: \mathbb{A}^n \to \mathbb{A}$ is said to be {\it t-continuous} if it is continuous with respect to the t-topology.

\medskip\noindent
By Claim I in the proof of Proposition 2.5 in \cite{Pillay}, the t-topology agrees with the topology induced from $K^m$ except for finitely many points; that is there is a finite subset $X$ of $\mathbb A$ such that the topologies on $\mathbb A\setminus X$ induced from $\mathbb A$ and $K^m$ are equivalent.

%
% \begin{lemma}\textcolor{blue}{Let $G$ be any subgroup of $\mathbb A$. If $\mathbb A$ is t-compact, then $G$ is t-dense in a semialgebraic subgroup of $\mathbb A$.}
% \end{lemma}
% \begin{proof} \textcolor{blue}{If $G$ is finite, it is semialgebraic. So we can reduce to the case that $G$ is infinite. Let $H$ be the t-definably connected component of $\mathbb{A}$ containing the identity of the $\mathbb{A}$. By \cite{Pillay} Proposition 2.12, $H$ is of finite index in $\mathbb{A}$. Since $G$ is an infinite subgroup of $\mathbb A$, the intersection $G\cap H$ is infinite. Since $\mathbb A$ is t-compact, so is $H$. Hence $G\cap H$ is t-dense in $H$. Now let $a_1,...,a_n \in G$ be a maximal set of elements such that each $a_i$ is in a different coset of $H$. Then $G\cap (a_i \oplus H) = a_i \oplus (G\cap H)$. Since the group operation is t-continuous, $a_i \oplus (G\cap H)$ is t-dense in $a_i \oplus H$. Hence $G$ is t-dense in $H \cup \bigcup_{i=1}^{n} a_i \oplus H$.}
% \end{proof}

\medskip\noindent
In \cite{Pillay}, it is proven that $\mathbb A$ must be abelian-by-finite. By Theorem 1.1 of \cite{Edmundo-Otero} we also know that $\mathbb A$ has finitely many $n$-torsion elements for each $n>0$.

\medskip\noindent
{\it Throughout the rest of the paper, we assume that $(\mathbb A,\oplus)$ is t-connected.}

\medskip\noindent
By Proposition 2.12 of \cite{Pillay} it follows that $\mathbb A$ does not have any proper infinite definable subgroup. Combining this with the previous paragraph we get that $\mathbb{A}$ is abelian and divisible.
%Note that $\mathbb A(\R)$ becomes a real Lie group of dimension one. Hence it has finitely many $n$-torsion elements for each $n>0$.
%The structure we are interested in is $(\R,\Gamma)$ in the language $\lam(P)$.

\medskip\noindent
Note that the t-connectedness assumption is not very restrictive, because every group definable in $K$ is a finite union of cosets of its t-connected component.

\medskip\noindent
Let $\pi_1,\dots,\pi_m$ be the standard projections of $K^m$ onto $K$. For our purposes it is harmless to assume that there is $i\in\{1,\dots,m\}$ such that for every $a\in\mathbb A$ and $j=1,\dots,m$ we have that $\pi_j(a)$ is $\lam$-definable over $\pi_i(a)$. Moreover we take $i$ to be $1$ and we sometimes write $\pi$ instead of $\pi_1$.

\medskip\noindent
For convenience we also assume that $\mathbb A$ is definable over $\emptyset$ and for another real closed field $E$, we let $(\mathbb A(E),\oplus)$ denote the group definable in $E$ by the formulas defining $(\mathbb A,\oplus)$ in $K$.

\medskip\noindent
Let $k=(k_1,\dots,k_n)$ be a tuple of integers. Consider the group character
$$\chi_k:\mathbb A^n\to\mathbb A,\quad \chi_k(a_1,\dots,a_n):=k_1a_1\oplus\cdots\oplus k_na_n$$
and let $T_k$ denote the kernel of $\chi_k$.
% Also if $G$ is a subgroup of $\mathbb A$ and $n>0$, then we put
% $$nG:=\{ng:g\in G\},$$
% $$G[n]:=\{g\in G:ng=0\},\quad\text{the $n$-torsion subgroup of $G$}.$$

\medskip\noindent
% The Mann property is about the solutions of linear equations in a subgroup of the multiplicative group of a field and it does not have a direct analog in the present setting.
% However it is shown in \cite{vdDG} that for multiplicative subgroups of algebraically closed fields of characteristic zero the Mann property is equivalent to another diophantine property that generalizes to the new framework and that property proved to be useful in the papers \cite{zilber2} and \cite{BZ}.

\medskip\noindent
Fix $n>0$ and a tuple of distinct indeterminates $X=(X_1,\dots,X_{mn})$. We usually denote an element of the polynomial ring $K[X]$ by $p$ (possibly with subscripts) and if we want to make the variables precise, we write $p(X)$. In what follows we identify $K^{mn}$ with $(K^m)^n$. In particular, for $\alpha_1,\dots,\alpha_n\in K^m$ and a polynomial $p\in K[X]$, $p(\alpha_1,\dots,\alpha_n)$ means $$p\big(\pi_1(\alpha_1),\pi_2(\alpha_1),\dots,\pi_m(\alpha_1),\dots,\pi_1(\alpha_n),\dots,\pi_m(\alpha_n)\big).$$
In a similar fashion, for a subfield $L$ of $K$ and a subset $S$ of $\mathbb A$, $L(S)$ denotes the subfield
$L\big(\pi_1(S)\cup\cdots\cup\pi_m(S)\big)$
of $K$ and  $L[S]:=L\big[\pi_1(S)\cup\cdots\cup\pi_m(S)\big].$

\medskip\noindent
Finally for a polynomial $p$ as above we put
$$V(p):=\{\alpha\in K^{mn}: p(\alpha)=0\},$$
the {\it zero set of} $p$ (in $K$).

\medskip\noindent
{\it In the rest of this section $L$ is a subfield of $K$ and $G$ is a subgroup of $\mathbb A$.}

\begin{definition}
We say that $G$ has the {\it Mordell-Lang property over $L$} if for every $n>0$ and for every polynomial $p\in L[X]$, there are $g_1,\dots,g_t\in G^n$ and $k_1,\dots,k_t\in\Z^n$  such that
$$V(p)\cap G^n=\bigcup_{i=1}^{t}g_i\oplus(T_{k_i}\cap G^n).$$

\end{definition}
% It is worth saying that this is the diophantine property that is forced on the groups in the papers \cite{zilber2} and \cite{BZ}.
\noindent
The reason for calling this property with this name is that it is the conclusion of a conjecture of Lang generalizing a conjecture of Mordell for abelian varieties. We refer the reader to \cite{NTIII} for the precise statement of the conjecture and its history.

\medskip\noindent
We proceed to show that if $G$ has the Mordell-Lang property over $\Q$, then it has the Mordell-Lang property over $K$.

\begin{lemma}
 Let $L$ contain $\Q(G)$ and suppose that $G$ has the Mordell-Lang property over $L$. Then $G$ has the Mordell-Lang property over $L^{\rc}$.
\end{lemma}

\begin{proof}
Let $\alpha\in K$ be algebraic over $L$ of degree $d>1$. It suffices to show that $G$ has the Mordell-Lang property over $L(\alpha)$. Take a polynomial $p\in L[\alpha][X]$.  Write $$p=p_0+p_1\alpha +\cdots+p_{d-1}\alpha^{d-1},$$
where $p_i\in L[X]$ for $i=0,1,\dots,d-1$. Then for $g=(g_1,\dots,g_n)\in G^n$
$$p(g)=0\Longleftrightarrow p_i(g)=0\text{ for each }i\in\{0,1,\dots,d-1\}.$$
Therefore
$$V(p)\cap G^n=\bigcap_{i=0}^{d-1} V(p_i)\cap G^n=V(p_0^2+\dots+p_{d-1}^2)\cap G^n.$$
By the Mordell-Lang property over $L$, we know that $V(p_0^2+\dots+ p_{d-1}^2)\cap G^n$ is a finite union of cosets of kernels of $\chi_k$ in $G^n$, thus so is $V(p)\cap G^n$.
\end{proof}

% \medskip\noindent
% We need the following piece of notation in the next step: For $s\in\N$ and $\vec i=\big(i(1),\dots,i(s)\big)\in\N^s$, $|\vec i|$ denotes $i(1)+\dots+i(s)$, and for a tuple $Y=(Y_1,\dots,Y_s)$ of distinct indeterminates, $Y^{\vec i}$ is the monomial $Y_1^{i(1)}Y_2^{i(2)}\cdots Y_s^{i(s)}$; likewise for $\alpha=(\alpha_1,\dots,\alpha_s)\in K^s$, $\alpha^{\vec i}$ means $\alpha_1^{i(1)}\alpha_2^{i(2)}\cdots\alpha_n^{i(s)}$.
\medskip\noindent
We need the following notation in the next step: For $s\in\N$ and a tuple $i=\big(i(1),\dots,i(s)\big)\in\N^s$, $|i|$ denotes $i(1)+\dots+i(s)$, and for a tuple $Y=(Y_1,\dots,Y_s)$ of distinct indeterminates, $Y^{i}$ is the monomial $$Y_1^{i(1)}Y_2^{i(2)}\cdots Y_s^{i(s)}.$$
Likewise for $\alpha=(\alpha_1,\dots,\alpha_s)\in K^s$, $\alpha^{i}$ means $\alpha_1^{i(1)}\alpha_2^{i(2)}\cdots\alpha_s^{i(s)}$.

\begin{lemma}
 Let $G$ have the Mordell-Lang property over $\Q$. Then $G$ has the Mordell-Lang property over $\Q(G)$.
\end{lemma}

\begin{proof}
Take a polynomial $p\in \Q[G][X]$ of degree $d$. Write
$$p=\sum_{|i|\leq d}\sum_{j}a_{i,j}\, g^{j}X^{i},$$
where $i$ and $j$ run through elements of $\N^{mn}$ and $\N^{mt}$ respectively, $a_{i,j}\in\Q$, and $g=(g_1,\dots,g_t)\in G^t$.

Let $Y=(Y_1,\dots,Y_{mt})$ be a tuple of indeterminates different than $X$ and put $q(X,Y)=\sum_{|i|\leq d}\sum_{j}a_{i,j}\, X^{i}Y^{j}\in\Q[X,Y]$.
For $g^*\in G^n$ we have
$$p(g^*)=0\Longleftrightarrow q(g^*,g)=0.$$
Now the result follows since $G$ has the Mordell-Lang property over $\Q$.
\end{proof}

\begin{proposition}
 Let $G$ have the Mordell-Lang property over $\Q$. Then $G$ has the Mordell-Lang property over $K$.
\end{proposition}

\begin{proof}
Let $E\subseteq K$ be a finitely generated extension of $\Q(G)^{\rc}$, and take a transcendence basis $\alpha=(\alpha_1,\dots,\alpha_t)$ of $E$ over $\Q(G)^{\rc}.$

Take a polynomial $p\in E[X]$, and write
$$p=\sum_{i} p_{i}\alpha^{i},$$
where $i=\big(i(1),\dots,i(t)\big)$ runs through elements of $\N^{t}$ such that $|i|\leq s$ for some $s\in\N$ and  $p_{i}\in\Q(G)^{\rc}[X]$.

Now it is easy to see that for $g\in G^{n}$
$$p(g)=0\Longleftrightarrow p_i(g)=0\text{ for each }i.$$
Hence $V(p)\cap G^n$ is of the desired form since $G$ has the Mordell-Lang property over $\Q(G)^{\rc}$ by the previous two lemmas.
\end{proof}

\medskip\noindent
{\it From now on we assume that $G$ has the Mordell-Lang property over $\Q$.} As a consequence of the proposition above it is harmless to simply say that $G$ has the Mordell-Lang property.

\medskip\noindent
For a subset $S$ of $K[X]$, let
$$V(S):=\bigcap_{p\in S}V(p).$$
Let $C=K(\sqrt{-1})$ be the algebraic closure of $K$ and identify $C$ with $K^2$ in the usual way. We get the following consequence of the Mordell-Lang property.

\begin{corollary}\label{ML-general}
Let $X\subset C^{mn}$ be definable in the field $C$. Then $X\cap G^n$ is definable in the group $(G,\oplus)$.
\end{corollary}

\begin{proof}
By quantifier elimination for algebraically closed fields, it suffices to show that for every $S\subseteq K[X]$, there are $s,t \in \mathbb{N}$ and $g_{i,j}\in G^n$ and $k_{i,j}\in\Z^n$ for $i=1,\dots,s$ and $j=1,\dots,t$ such that
$$V(S)\cap G^n=\bigcap_{i=1}^{s}\bigcup_{j=1}^{t}g_{i,j}\oplus(T_{k_{i,j}}\cap G^n).$$
This easily follows from the Mordell-Lang property combined with Hilbert's Basis Theorem.
\end{proof}

\begin{remark}
 Note that we cannot get this result with $K$ in the place of $C$.  Vaguely speaking, it is not possible to define the trace of ordering on the group $G$.
\end{remark}

\subsection{The main lemma}
We prove an analog of Lemma 5.12 in \cite{vdDG}, which is the most useful consequence of the Mordell-Lang property. We take this opportunity to introduce some more algebraic notations and conventions.

\medskip\noindent
Let $G'$ be a subgroup of $\mathbb A$ containing $G$ and $g',g_1',\dots,g_n'$ elements of $G'$.  We say that $g'$ is {\it algebraic over} $L$ if $\pi(g')$ is algebraic over $L$, and similarly $g_1',\dots,g_n'$ are {\it algebraically dependent over }$L$ if $\pi(g_1'),\dots,\pi(g_n')$ are algebraically dependent over $L$. Also we say that $g_1',\dots,g_n'$ are {\it linearly dependent modulo} $G$ if there is $k\in\Z^n\setminus\{(0,\dots,0)\}$ such that $\chi_k(g_1',\dots,g_n')\in G$.

\medskip\noindent
% Now let $p(X)\in K[X]$. Say
% $$V(p)\cap G^n=g_1\oplus T_1\cup\cdots\cup g_s\oplus T_s$$
% where $g_1,\dots,g_s\in G^n$ and $T_i=T_{k_i}$ with $k_i\in\Z^n$ for $i=1,\dots,s$.
We say that $G'$ {\it satisfies the same Mordell-Lang conditions as} $G$ if for every polynomial $p\in \Q[X]$
$$V(p)\cap (G')^n=g_1\oplus (T_{k_1}\cap (G')^n)\cup\cdots\cup g_t\oplus (T_{k_t}\cap (G')^n),$$
with $g_1,\dots,g_t\in G^n$ and $k_1,\dots,k_t\in\Z^n$.

\medskip\noindent
Note that if $G'$ satisfies the same Mordell-Lang conditions as $G$, \label{ML-conditions} then $G'$ has the Mordell-Lang property as well.

% \medskip\noindent
% The next result is an analogue of Lemma 5.12 in \cite{vdDG}.

\begin{lemma}\label{dependence}
 Let $G'$ be a subgroup of $\mathbb A$ containing $G$ and suppose that $G'$ satisfies the same Mordell-Lang conditions as $G$. Then $g_1',\dots,g_n'\in G'$ are linearly dependent modulo $G$ whenever they are algebraically dependent over $\Q(G)$.

% Then we have:
%
% \noindent $\operatorname{(1)}$ If $g^*$ is algebraic over $\Q(G)$, then $dg^*\in G$ for some $d\geq 1$.
%
% \noindent $\operatorname{(2)}$ If $g_1^*,\dots g_n^*$ are algebraically dependent over $\Q(G)$, then they are linearly dependent \textcolor{blue}{modulo} $G$.
\end{lemma}

\begin{proof}
 We just prove the case $n=1$ and the general case can be proven using similar arguments. So let $g'\in G'$ be algebraically dependent over $\Q(G)$. Let $X=(X_1,\dots,X_m)$ and $Y=(Y_1,\dots,Y_{mt})$ be tuples of distinct indeterminates, and take a polynomial $p(X_1,Y)\in \Q[X_1,Y]$ and $h\in G^t$ such that $p(\pi(g'),h)=0$ and $p(X_1,h)$ is not the zero polynomial. Then considering $p(X_1,Y)$ as an element of $K[X,Y]$ and using the Mordell-Lang condition for $n=t+1$ we get $k\in\Z^{t+1}$ and $h'\in G^{t+1}$ such that $(g',h)\in h'\oplus (T_k\cap (G')^{t+1})$. Note that $k_1\neq 0$, because otherwise $p(\pi(g''),h)=0$ for every $g''\in G'$ and hence $p(X_1,h)$ is the zero polynomial. Now it is easy to see that $k_1g'\in G$.
\end{proof}

\subsection{Smallness revisited}
The aim of this subsection is to prove Corollary \ref{smallnesscor} below, which is used in Section~\ref{modeltheory} in a very essential way. That result is a consequence of an abstract condition called {\it smallness}, which in turn is satisfied by the groups with the Mordell-Lang property (see Proposition \ref{small} below).

\medskip\noindent
Here we define smallness only in the setting of fields; in Section \ref{opencoresection}, we define it in a more general setting.
First we  recall some notations: For a positive integer $l$, an {\it $l$-valued map}, denoted as $f:X\stackrel{l}{\longrightarrow} Y$, is a map from $X$ to $\Cal P(Y)$ such that $|f(x)|\leq l$ for every $x\in X$; and such a map is {\it definable} in a given structure $\Cal M$ if its {\it graph}
$$\{(x,y)\in X\times Y: y\in f(x)\}$$
is definable in $\Cal M$. For $A\subseteq X$, we let $f(A):=\bigcup_{a\in A}f(a)$.

\begin{definition}
Let $E$ be a field. A subset $X$ of $E^s$ is called {\it large} if there is a map $f:E^{sn}\longrightarrow E$ definable in the field $E$ such that $f(X^n)=E$; otherwise we say that $X$ is {\it small}.
\end{definition}

\begin{remarks}
(1) Smallness is an elementary property of the pair $(E,X)$ construed as a structure in the language of rings expanded by an $s$-ary relation symbol.

\smallskip\noindent
(2) If $E$ is an ordered field or is an algebraically closed field, then $X$ is large if and only if there is multi-valued map $f:E^{sn}\stackrel{l}{\longrightarrow} E$ definable in the field $E$ such that $f(X)=E$. It is easy to see this when $E$ is an ordered field and for algebraically closed fields see Lemma 2.4 in \cite{vdDG} (note that the definition of large is different in that paper).
\end{remarks}

%  It is also easy to see that if $E$ is a real closed field, then a set $X\subseteq E^s$ is large if and only if there is a usual function $g:E^{sn'}\to E$ definable in $E$ such that $g(X^{n'})=E$ (This is still true if $E$ is an algebraically closed field, but requires more work; see Lemma 2.4 of \cite{vdDG}).

% \medskip\noindent
% \textcolor{blue}{{\bf Remark.} Let $F\subseteq E$ be fields. Then it is easy to see that $F$ is large in $E$ if and only if there is a natural number $l$ such that the degree over $F$ of each element of $E$ is less than $l$. Then using Lemma 3.1 of \cite{Keisler}, $F$ is large in $E$ if and only if either $F=E$ or $F$ is a real closed field with $F(\sqrt{-1})=E$.}

%  the following are equivalent:
% \begin{enumerate}
% \item $F$ is small in $E$;
% \item $[E:F] >2$;
% \item $[E:K]=\infty$.
% \end{enumerate}

\medskip\noindent
We first mention a result that has been neglected in \cite{BEG}.
It must be known by many people, but we could not find a reference for it anywhere. So we include a proof as well.

\medskip\noindent
Remember that in the beginning of this section we fixed $K$ to be a  real closed field and $G$ a subgroup of a t-connected one dimensional group definable in $K$ with the Mordell-Lang property.  Let $C=K(\sqrt{-1})$.

\begin{lemma}\label{codensity}
If $X\subseteq K^s$ is small in $C$. Then $X$ is small in $K$.
%\textcolor{blue}{In particular any proper subfield of $E$ is small in $E$.}
\end{lemma}

\begin{proof}
By the first of the remarks above we may assume that $K$ is $\aleph_0$-saturated.  Also by the second remark, it is enough to prove the following:

\smallskip\noindent{\it Claim.} Let $f : K^s \to K$ be definable in the field $K$. Then there is a multi-valued function $\tilde f : C^s\stackrel{l}{\longrightarrow} C$ definable in the field $C$ such that $f(\alpha)\in\tilde f(\alpha)$ for each $\alpha\in K^s$.

\smallskip\noindent{\it Proof of the claim.} Suppose that $f$ is definable over $B\subseteq K$ and let $\alpha\in K^s$. Then $f(\alpha)$ is in the definable closure in $K$ of $B\cup\{\alpha\}$. Hence $f(\alpha)$ is in the algebraic closure in $C$ of $B\cup\{\alpha\}$. Let this be witnessed by a formula $\phi(x,y)$ in the language of rings; that is to say that
\begin{equation}\label{claimform}
C\models \phi(\alpha,f(\alpha)) \text{ and } |\{ y \in C :  C \models \phi(\alpha,y)\}| < \infty.
\end{equation}
By quantifier elimination for algebraically closed fields, the second part of \eqref{claimform} is expressible by a formula in the language of rings. By saturation of $K$, there are formulas $\phi_1,\dots,\phi_t$ in the language of rings such that for all $\alpha \in K^s$ there is $i \in \{1,\dots,t\}$ such that \eqref{claimform} holds with $\phi$ in the place of $\phi_i$.\\
For $\alpha \in C$, let $I_{\alpha}$ be the set
$$
\{ i \in \{1,\dots,t\} \ : \ |\{ y \in C : C\models\phi_i(\alpha,y) \}|<\infty\}.
$$
Now define a multi-valued function $\tilde{f}:C^s \to C$ by
\begin{equation*}
\tilde{f}(x) := \left\{
                  \begin{array}{ll}
                    \bigcup_{i \in I_x} \{ y \in C : C \models \phi_i(x,y) \}, & \hbox{ if $I_x \neq \emptyset$,}\\
                    \{0\}, & \hbox{otherwise.}
                  \end{array}
                \right.
\end{equation*}

%\textcolor{green}{ Suppose that $f$ is definable over $B\subseteq K$, say by a formula $\phi(x,y)$ in the language of %rings. Let $\alpha\in K^s$, then $f(\alpha)$ is in the definable closure in $K$ of $B\cup\{\alpha\}$. Then by ???, %$f(\alpha)$ is in the algebraic closure in $C$ of $B\cup\{\alpha\}$. Let this be witnessed by a formula %$\phi_{\alpha}(x,y)$ in the language of rings; that is to say that $C\models \phi_{\alpha}(\alpha,f(\alpha))$ and %there are finitely many, say $l_\alpha$ many, $\alpha'\in C$ such that $C\models\phi_{\alpha}(\alpha,\alpha')$. Let  %$U_{\alpha}$ be the set of $\beta\in C^s$ such that there are $l_\alpha$ many $\beta'\in C$ with %$C\models\phi_{\alpha}(\beta,\beta')$ and $C\models\forall y(\phi(\beta,y)\rightarrow\phi_{\alpha}(\beta,y))$.
%Note that $U_{\alpha}$ is definable in $C$ and that $\alpha\in U_{\alpha}$. Now by saturation of $K$, there are %$\alpha_1,\dots,\alpha_t\in C^s$ such that $K^s \subseteq U_{\alpha_1}\cup\cdots\cup U_{\alpha_t}$, which gives the %desired multi-valued function $\tilde f$.}
\end{proof}

% \medskip\noindent
% \textcolor{blue}{Let $C$ be an algebraically closed field and $X\subseteq C^s$ for some $s\in\N$. Let $X_{\operatorname{ind}}$ be the first order structure with the underlying set $X$ and definable sets $X^n\cap Z$ where $Z\subset C^{sn}$ is definable in the field $C$.}
%
% \textcolor{blue}{
% \begin{proposition}
% A subset $Y\subseteq X^n$ is definable in the pair $(C,X)$ if and only if it is definable in $X_{\operatorname{ind}}$.
% \end{proposition}}
%
%
% \begin{proof}
% \textcolor{blue}{The right to left direction is clear. For the other direction let $(C',X')$ be a $|C|^+$-saturated elementary extension of $(C,X)$ and take $\alpha=(\alpha_1,\dots,\alpha_{sn})$ and $\beta=(\beta_1,\dots,\beta_{sn})$ from $(X')^s$ such that
% $$\tp_{X'_{\operatorname{ind}}}(\alpha/X)=\tp_{X'_{\operatorname{ind}}}(\beta/X).$$
% We need to show that $\alpha$ and $\beta$ have the same type over $C$ in the pair $(C',X')$.}
%

% \medskip\noindent
% \textcolor{blue}{By the assumption on $\alpha$ and $\beta$ we have a field automorphism of $\Q(X')$ over $\Q(X)$ taking $\alpha$ to $\beta$.  Since $\Q(X')$ and $C$ are linearly disjoint over $\Q(X)$, this automorphism extends to an automorphism of $C(X')$ over $C$ and hence to an automorphism of $C'$ over $C$. This last automorphism is indeed an automorphism of the pair $(C',X')$ taking $\alpha$ to $\beta$.}
% \end{proof}

% \medskip\noindent
% {\bf Remark.} Using Corollary \ref{ML-general}, it is easy to see that for $X\subseteq C^{mn}$ definable in the algebraically closed field $C$, the set $X\cap G^n$ is definable in the group $(G,\oplus)$.

\medskip\noindent
We need the following general model theoretic fact in the next proposition.

\medskip\noindent
{\bf Fact.} A field is interpretable in an abelian group only if it is finite.  For this, recall the well-known model theoretic results that every abelian group is one-based and that a group interpretable in a one-based structure has an abelian subgroup of finite index (see \cite{Pillay_GST}). Now the fact follows since $\operatorname{SL}_2(E)$ does not have an abelian subgroup of finite index for an infinite field $E$.

\medskip\noindent
Now we are ready to prove the following.

\begin{proposition}\label{small}
 The group $G$ is small in $K$.
\end{proposition}

\begin{proof}
By Lemma~\ref{codensity}, it is enough to show that $G$ is small in the algebraically closed field $C$.
For a contradiction let $f:C^{mn}\to C$ be definable in the field $C$ such that $f(G^n)=C$. Let $R\subseteq C^{2mn}$ be the equivalence relation on $C^{mn}$ defined as follows
$$R(x,y)\Longleftrightarrow f(x)=f(y),$$
and put $R_G:=R\cap G^{2n}$, which is definable in the group $(G,\oplus)$ by Corollary~\ref{ML-general}. Then $f$ gives a bijection between $G^{n}/R_G$ and $C$, and we carry the addition and multiplication on $C$ to $G^{n}/R_G$ using this bijection, which are interpretable in $(G,\oplus)$. This gives an infinite field interpretable in an abelian group, contradicting the fact above.
% Take a proper elementary extension $(K^*,G^*)$ of $(K,G)$. It suffices to show that \textcolor{blue}{$\Q(G^*)$ is small in $K^*$.
% By the previous corollary, it is enough to prove that
% %$\Q(G^*)$ is small in the algebraic closure $K^*(\sqrt{-1})$ of $K^*$. Using \cite{Keisler} it is enough to show that
%  $\Q(G^*)\neq K^*$.}
% %
% % Consider the diagram below where
% % all arrows are inclusions:
% % \begin{displaymath}
% % \xymatrix{\ &\ &K^*\\
% %  \ &K \ar[ru]  & \Q(\Gamma^*) \ar[u] \\
% % \ & \Q(\Gamma) \ar[ru] \ar[u] }
% % \end{displaymath}
%
% Take a subset $B$ of $G^*$ that is linearly independent \textcolor{blue}{modulo} $G$ and is maximal with respect to this property. Clearly $B\neq\emptyset$. By Lemma \ref{dependence} we have that $B$ is algebraically independent over $\Q(G)$, and hence $\Q(G,B)$ is a purely transcendental extension of $\Q(G)$. Then $\Q(G^*)$ cannot be real closed, as for every $g^*\in G^*$, there is a positive integer $d$ such that \textcolor{blue}{$dg^*\in(\Q(G,B))^m$}. Hence $\Q(G^*)\neq K^*$.
\end{proof}

\medskip\noindent
A consequence of smallness is the following.

\begin{corollary}\label{smallnesscor}
Let $f_1,\dots,f_l:K^{mn}\to K$ be definable in the ordered field $K$. Then $K\setminus \bigcup_{i=1}^{l}f_i(G^n)$ is dense in $K$.
% and let $\alpha_1,\alpha_2\in K$ be such that $\alpha_1<\alpha_2$. Then there is $\alpha\in(\alpha_1,\alpha_2)$ such that $\alpha\notin f(\Gamma^n)$.
\end{corollary}

\begin{proof}
Let $f:K^{mn}\to K$ be the $l$-valued map taking $\alpha\in K^{mn}$ to the set $\{f_1(\alpha),\dots,f_l(\alpha)\}$. Assume that a nonempty interval $I$ of $K$ is contained in $f(G^n)$. Take a function $g:K\to K$ definable in the ordered field $K$ that maps $I$ onto $K$. Now $(g\circ f)(G^n)=K$, contradicting the smallness of $G$.
\end{proof}

\section{Model theory} \label{modeltheory}
\noindent
Remember that $(\mathbb A,\oplus)$ is a one-dimensional t-connected group definable in a real closed field $K$ over $\emptyset$. As before for another real closed field $E$, we let $\mathbb A(E)$ denote the group definable in $E$ by the formulas defining $\mathbb A$ in $K$. As mentioned above, such a group is abelian, divisible and has finitely many $n$-torsion points for every $n>0$.

\medskip\noindent
Fix a subgroup $\Gamma$ of $\mathbb A(\R)$ with the Mordell-Lang property such that $|\Gamma/n\Gamma|$ is finite for every $n>0$.

\subsection{The theory}
Let $\lam(P)$ be the language $\lam$ of ordered rings expanded by an $m$-ary relation symbol $P$ (note that $m=2$ in the introduction). Also let $\lam(\Gamma)$ be the language $\lam$ augmented by constant symbols $\pi(\gamma)$ for each $\gamma \in \Gamma$ and let $\lam(P;\Gamma)$ the language $\lam(\Gamma)$ extended by $P$.
For simplicity of notation we denote $\lam(\Gamma)$-structures by $\big(K,(\gamma)_{\gamma\in\Gamma}\big)$, rather than $\big(K,(\pi(\gamma))_{\gamma\in\Gamma}\big)$; similarly $\big(K,G,(\gamma)_{\gamma\in\Gamma}\big)$ are $\lam(P;\Gamma)$-structures.
%   In order to axiomatize $(\R,\Gamma)$, we extend $\lam(P)$ to a language $\lam(P;\Gamma)$ by adding a name for each element of $\pi_1(\Gamma)$.

% \medskip\noindent
% Let $n>0$. As before $X=(X_1,\dots,X_{mn})$ is a tuple of distinct indeterminates. For every polynomial $p(X)\in\Q[X]$ fix finite sets $\{k_1,\dots,k_s\}\subseteq\Z^n$ and $\{\gamma_1,\dots,\gamma_s\}\subseteq\Gamma^n$  such that
% $$V(p)\cap\Gamma^n=\bigcup_{i=1}^{s}\gamma_i\oplus T_{k_i}.$$
% Then the {\it Mordell-Lang axiom for} $p$ is the $\lam(P;\Gamma)$-sentence
% $$\forall x_1\cdots\forall x_{mn}\big[\bigwedge_{j=0}^{n-1}P(x_{jm+1},\dots,x_{jm+m})\wedge p(x_1,\dots,x_{mn})=0\rightarrow$$ $$\bigvee_{i=1}^{s}\chi_{k_{i}}\big((x_1,\dots,x_m)\ominus\gamma_{i,1},\dots,(x_{mn-m+1},\dots,x_{mn})\ominus\gamma_{i,n}\big)=0\big].$$

\medskip\noindent
Let $T$ be the $\lam(\Gamma)$-theory of $\big(\mathbb R,(\gamma)_{\gamma \in \Gamma}\big)$ and
let $T(\Gamma)$ be the $\lam(P;\Gamma)$-theory extending $T$ whose models are of the form $\big(K,G,(\gamma)_{\gamma\in\Gamma}\big)$ satisfying the following:
\begin{enumerate}
 \item $G$ is a t-dense subgroup of $\mathbb A(K)$,
%\item $\Gamma$ is a pure subgroup of $G$,
\item for every $n>0$ and $g\in G$, if $ng\in\Gamma$, then $g\in\Gamma$,
\item for every $n>0$, $|G/nG|=|\Gamma/n\Gamma|$,
\item $G$ satisfies the same Mordell-Lang conditions as $\Gamma$ (see page \pageref{ML-conditions}).
\end{enumerate}

%\textcolor{blue}{As noted in the previous section the topology on $\mathbb A(K)$ coincides with the topology induced %from $K^m$ except finitely many points. Moreover these points belong to $\Gamma$. Therefore the first condition is %expressible in the language $\lam(P;\Gamma)$ as an axiom-scheme.}
\medskip\noindent
Using Proposition 2.5 of \cite{Pillay} once again, we get a finite subset $S$ of $\mathbb A(K)$ such that a subset $X$ of $\mathbb A(K)$ is t-dense if and only if $X\setminus S$ is dense in $\mathbb A(K)\setminus S$. Hence condition (1) is first order in the language $\lam(P;\Gamma)$.
It is easy to see that conditions (2) and (3) are also first order in the language $\lam(P;\Gamma)$; for the last one we fix $\gamma_1,\dots,\gamma_t\in\Gamma^n$ and $k_1,\dots,k_t\in\Z^n$ for a given polynomial $p\in\Q[X]$ such that that
$$V(p)\cap\Gamma^n=\bigcup_{i=1}^t \gamma_i\oplus (T_{k_i}\cap \Gamma^n),$$
and consider the formula
\begin{align}\forall x_1\cdots\forall x_{mn}\bigwedge_{j=0}^{n-1}P(x_{jm+1},\dots,x_{jm+m})\rightarrow
\big[ p(x_1,\dots,x_{mn})=0\nonumber\\
\leftrightarrow\bigvee_{i=1}^{t}\chi_{k_{i}}\big((x_1,\dots,x_{m}),\dots,(x_{mn-m+1},\dots,x_{mn})\big)=\chi_{k_i}(\gamma_i)\big]\nonumber.
\end{align}

\medskip\noindent
Note that if $\Gamma$ is t-dense in $\mathbb A(\R)$, then $(\R,\Gamma)$ is a model of $T(\Gamma)$. We are proceeding to show that $T(\Gamma)$ is complete in that case. We achieve that by constructing a back-and-forth system between models of $T(\Gamma)$.
The same back-and-forth system gives that $T(\Gamma)$ has quantifier elimination up to formulas of the form
$$\exists y_1\cdots\exists y_{mn}\big(\bigwedge_{j=0}^{n-1}P(y_{mj+1},\dots,y_{mj+m})\wedge\phi(x,y_1,\dots,y_{mn})\big)$$
where $x$ is a tuple of distinct variables and $\phi$ is a formula in the language $\lam(\Gamma)$.

\medskip\noindent
In the rest of this section $\big(K,G,(\gamma)_{\gamma\in\Gamma}\big)$ ranges over models of $T(\Gamma)$, and we denote them simply by $(K,G)$.

% \medskip\noindent
% \textcolor{blue}{Since $G$ is dense in $\mathbb A$, it follows that for $n>0$, the group $nG$ is dense in the semialgebraic set $n\mathbb A$.  Thus each coset of $nG$ in $G$ is densely included in a semialgebraic set.}

\medskip\noindent
For $k=(k_1,\dots,k_n)\in\Z^n$ and $e\in\N$, define
$$D_{k,e}:=\chi_k^{-1}\big(eG\big)\cap G^n.$$
Note that $D_{k,e}$ is a subset of $G^n$ definable in $\lam(P)$ and that $(eG)^n\subseteq D_{k,e}$. Hence we have that $D_{k,e}$ is of finite index in $G^n$ as $eG$ is of finite index in $G$. Thus both $D_{k,e}$ and $G^n\setminus D_{k,e}$ are finite unions of cosets (in $G^n$) of $(eG)^n$.  Using the fact that $eG\cap e'G=\lcm(e,e')G$ for $e,e'\in\N$, we get the following consequence.

\begin{lemma}\label{unionofcosets}
Let $n>0$, $k_1,\dots,k_s\in\Z^n$ and $e_1,\dots,e_t\in\N$.
Then every boolean combination (in $G^n$) of cosets of $D_{k_i,e_j}$ in $G^n$ with $i\in\{1,\dots,s\}$ and $j\in\{1,\dots,t\}$ is a finite union of cosets of $(lG)^n$, where $l$ is the lowest common multiple of $e_1,\dots,e_t$.
\end{lemma}
%
% \begin{proof}
% $\end{proof}
\medskip\noindent
{\bf Remark.} Note that the coset representatives can be chosen from $\Gamma^n$ by axiom (3).
Moreover, $l$ in the lemma depends only on $e_1,\dots,e_t$ and not on $G$ or $k_1,\dots,k_s$.

\begin{lemma}\label{denseinsemi} Let $\gamma\in \mathbb A(K)^n$, $k\in\Z^n$ and $e\in\N$. Then $\gamma\oplus D_{k,e}$ is t-dense in $\mathbb A(K)^n$.
\end{lemma}
\begin{proof} Since $G$ is t-dense in $\mathbb A(K)$ and multiplication by $e$ is a t-continuous map on $\mathbb A(K)$, we have that $(eG)^n$ is t-dense in $(e\mathbb A(K))^n$. Since $\mathbb A(K)^n$ is divisible, $(eG)^n$ is t-dense in $\mathbb A(K)^n$. Since  $(eG)^n\subseteq D_{k,e}$, we have that $D_{k,e}$ is t-dense in $\mathbb A(K)^n$. Since addition is t-continuous, $\gamma \oplus D_{k,e}$ is t-dense $\mathbb A(K)^n$. \end{proof}

\medskip\noindent
Recall that a subgroup $H$ of $G$ is called \emph{pure} if $nG\cap H=nH$ for every $n>0$.  For a subset $X$ of $G$ we let $\<X\>_G$ be the subgroup of $G$ generated by $X$ and we let $[X]_G$ be the subgroup of $G$ consisting of $g$ such that $ng\in\<X\>_G$ for some $n>0$. When the ambient group $G$ is clear from the context, we omit $G$ from both of these notations.
%For a pure subgroup $H$ of $G$ and a subset $A$ of $G$, let $H_G\<A\>$ be the set of $g\in G$ such that $ng$ is in the subgroup of $G$ generated by $H$ and $A$
%for some $n > 0$; that is there are $h \in H$, $a_1 ,\dots , a_t \in A$, and $k_1 ,\dots , k_t \in
%\Z$, such that $ng = h \oplus k_1a_1 \oplus\dots\oplus k_t a_t$. When the ambient group $G$ is clear
%from the context, we omit $G$ from the notation. Note that $H_G\<A\>$ is the smallest pure subgroup of $G$ containing both $H$ and $A$.

\medskip\noindent
We prove some lemmas that will be useful in the rest of the section.

\begin{lemma}\label{submainlemma}
Let $H$ be a pure subgroup of $G$ containing $\Gamma$ and let $g\in G$. Then
$$\big(\Q(H,g)^{\rc}\big)^m\cap G=[H\cup\{g\}]$$
\end{lemma}

\begin{proof}
It is easy to see that $[H\cup\{g\}]\subseteq\big(\Q(H,g)^{\rc}\big)^m\cap G$.
% since all the torsion elements of $G$ are algebraic over $\Q$.
Now take $g'\in(\Q(H,g)^{\rc})^m\cap G$. Since $G$ satisfies the same Mordell-Lang conditions as $H$, we can apply Lemma \ref{dependence} to get that $g$ and $g'$ are linearly dependent modulo $H$. Thus $g'\in [H\cup\{g\}]$.
\end{proof}

\medskip\noindent
We can strengthen this lemma as follows.

\begin{lemma}\label{mainlemma} Let $H,g$ be as in the previous lemma and let $X$ be a subset of $K$ algebraically independent over $\pi(G)$.
Then
% there is an
% $\lam(\Gamma)$-$\emptyset$-definable function $f$ such that
% $g \in G$ and
\begin{equation*}
\big(\Q(X,H,g)^{\rc}\big)^m \cap G=[H \cup\{g\}].
\end{equation*}
\end{lemma}

\begin{proof}
% It is easy to see that $ H_G\langle g \rangle$ is a subset of $\big(\acl(A,\pi_1(g),\pi_1(H)\big)^m \cap G$.  For the other inclusion, consider an element $k \in \big(\acl(\pi_1(g),\pi_1(H)\big)^m\cap G$. We have $h_1,\dots,h_l \in H$ such that for every $i=1,\dots,m$ and polynomial $p_i(X_1,\dots,X_{l+2})$ with coefficients in $\mathbb{Q}$ such that
% \begin{equation*}
% p_i(\pi_i(k),\pi_1(g),\pi_1(h_1),\dots,\pi_1(h_l))=0.
% \end{equation*}
% Let $p \in \mathbb{Q}[X_1,\dots,X_{(l+2)m}]$ be defined by
% \begin{equation*}
% \sum_{i=1}^{m} p_i(X_i,X_{m+1},X_{2m+1},\dots,X_{lm+1})^2.
% \end{equation*}
% This polynomial $p$ is non-trivial and
% \begin{equation*}
% p(k,g,h_1,\dots,h_l)=0.
% \end{equation*}
% Since $k,g,h_1,\dots,h_l \in G$, we get by (5) that there is $\gamma \in \Gamma$ such that
% \begin{equation*}
%  \bigoplus_{i=1}^{l} n_i h_i \oplus n_{l+1} g \oplus n_{l+2} k =\gamma.
% \end{equation*}
% Since $\Gamma$ is a subset of $H$, this implies that
% \begin{equation*}
% n_{l+2} k \in H_G\langle g \rangle.
% \end{equation*}
% Since $H_G\langle g \rangle$ is pure in $G$, we have $k \in H_G\langle g \rangle$. Hence
% \begin{equation}\label{mlcor}
% (\acl(\pi_1(g),\pi_1(H)))^m\cap G \subseteq H_G\langle g \rangle.
% \end{equation}
By the previous lemma, all we need to show is
\begin{equation}\label{mlcor2} \big(\Q(X,H,g)^{\rc}\big)^m\cap G \subseteq
\big(\Q(H,g)^{\rc}\big)^m\cap G.
\end{equation}
Let $g' \in \big(\Q(X,H,g)^{\rc}\big)^m \cap G$.
% \sout{So $\pi(g')\in \Q(X,g,H)^{\rc}$.}
Let
$X'$ be a minimal subset of $X$ such that $g' \in\big(\Q(X',g,H)^{\rc}\big)^m$. For
a contradiction, suppose that $X'$ is nonempty and let $x \in X'$.
By minimality of $X'$, we have that $g' \notin (\Q(X'\setminus\{x\},H,g)^{\rc})^m$. But then
the Steinitz Exchange Principle implies that $x \in
\Q(X'\setminus\{x\},H,g,g')^{\rc}$. Since $g,g' \in G$, we get that
\begin{equation*}
x \in \Q(X'\setminus\{x\},G)^{\rc}.
\end{equation*}
This contradicts with the assumption that $X$ is algebraically independent over
$\pi(G)$. Hence $X'$ is empty and $g' \in \big(\Q(H,g)^{\rc}\big)^m \cap G$.
\end{proof}

\begin{corollary}\label{zerodefing} Let $g\in G$. If $g$ is $\lam(\Gamma)$-$\emptyset$-definable, then $g \in \Gamma$.
\end{corollary}
\begin{proof} Using Lemma \ref{submainlemma}, we have $[\Gamma]=\big(\Q(\Gamma)^{\rc}\big)^m\cap G$. But $[\Gamma]=\Gamma$ by axiom $(2)$.
%
% By (2), $\Gamma$ is pure in $G$. With Lemma \ref{mainlemma}, this implies that $\Gamma = (\acl{\pi(\Gamma)})^m\cap G$. Since $\acl{\pi(\Gamma)}$ is just the set of $\lam(\Gamma)$-$\emptyset$-definable elements of $M$, we have that $g$ is in $\Gamma$.
\end{proof}

\begin{lemma}\label{projectionlemma} Suppose that $(K,G)$ is $|\Gamma|^+$-saturated.
Let $S$ be an $\lam(\Gamma)$-$\emptyset$-definable subset of $\mathbb A(K)$ and let $D$ be a t-dense subset of $G$. Then  the projection $\pi\big((D\setminus\Gamma) \cap S\big)$ is dense in the interior of $\pi(S)$.
\end{lemma}
\begin{proof} Let $Y\subseteq \mathbb A(K)$ be the finite set of points where the t-topology does not agree with the induced topology from $K^m$. Hence $D\setminus Y$ is dense in $\mathbb A(K) \setminus Y$.
By the construction of the t-topology in \cite{Pillay}, every point of $Y$ is $\lam(\Gamma)$-$\emptyset$-definable. Hence $G\cap Y \subseteq \Gamma$ by Corollary \ref{zerodefing}.

\medskip\noindent
By o-minimality, $\pi(S)$ is a finite union of open intervals and points. So suppose there is $g \in D$ such that $\pi(g)$ is one of these points. This implies that $\pi(g)$ is $\lam(\Gamma)$-$\emptyset$-definable, since $S$ is. Then $g\in\Gamma$ by Corollary \ref{zerodefing}. Hence $\pi\big((D\setminus \Gamma)\cap S\big)$ is in the interior of $\pi(S)$.

\medskip\noindent
Since $(K,G)$ is $|\Gamma|^+$-saturated, $\pi(\Gamma)$ is discrete in $K$. Because $D\setminus Y$ is dense in $\mathbb A(K)\setminus Y$ and $Y\cap D\subseteq \Gamma$, we get that $\pi((D\setminus \Gamma) \cap S)$ is dense in the interior of $\pi(S)$.
\end{proof}

%\begin{lemma}\label{projectionlemma} Let $f:K\to K^{m-1}$ be an  $\lam(\Gamma)$-$\emptyset$-definable function and $D$ a dense subset of $G$. Then $\pi((D\setminus\Gamma) \cap \gr(f))$ is dense in the interior of $\pi(\mathbb A(K) \cap \gr(f))$.
%\end{lemma}
%\begin{proof} First we show that $\pi((D\setminus\Gamma) \cap \gr(f))$ is a subset of the interior of $\pi(\mathbb A(K) \cap \gr(f))$. Since $\pi(\mathbb A(K) \cap \gr(f))$ is a semi-algebraic set, it is a finite union of open intervals and points. It is just left to show that if any of these finitely many points is in $\pi(D)$, then it is in $\pi(\Gamma)$. So suppose $\pi(g) \in \pi(D)$ is one of these points. This implies that $g$ is $\lam(\Gamma)$-$\emptyset$-definable, since $\pi(\mathbb A(K) \cap \gr(f))$ is. Then $g\in\Gamma$ by Corollary \ref{zerodefing}.

%\medskip\noindent We finish the proof by showing that $\pi((D\setminus\Gamma) \cap \gr(f))$ is dense in the interior of $\pi(\mathbb A(K) \cap \gr(f))$. Let $I$ be any subinterval of $\pi(\mathbb A(K) \cap \gr(f))$. We just need to find $g\in D$ with $\pi(g) \in I$. By the Monotonicity Theorem, we can assume that $f$ is continuous on $I$. Hence the graph of $f$ restricted to $I$ is an open subset of $\mathbb A(K)$. Since $G$ is dense in $\mathbb A(K)$ and $D$ is dense in $G$, there is an element $g \in D$ such that $g$ is in the graph of $f$ restricted to $I$. Hence for this $g$, we have $\pi(g) \in I$.
%\end{proof}

\subsection{Back-and-forth and completeness} Fix two $|\Gamma|^+$-saturated models $(K,G)$ and $(K',G')$ of $T(\Gamma)$, and
let $\mathcal{S}$ be the collection of $\lam(P; \Gamma)$-isomorphisms
$$\beta:\big(\Q(X,H)^{\rc},H\big)\to\big(\Q(X',H')^{\rc},H'\big)$$
% that are partial $\lam(\Gamma)$-elementary maps from $K$ to $K'$ such that
% there exist
where $H$ and $H'$ are pure subgroups of cardinality at most
  $|\Gamma|$ of $G$ and $G'$ containing $\Gamma$ and $X$ and $X'$ are finite subsets of $K$ and $K'$ that are algebraically independent over $\Q(G)$ and $\Q(G')$ respectively and $\beta(X)=X'$.
% \begin{itemize}
%   \item $X$ and $X'$ are finite subsets of $K$ and $K'$ respectively,
%   \item $H$ and $H'$ are pure subgroups of cardinality at most
%   $|\Gamma|$ of $G$ and $G'$ containing $\Gamma$.
% \end{itemize}
% such that
%
% \begin{enumerate}
%   \item $\beta$ is an $\lam(P; \Gamma)$-isomorphism between $(\mathbb{Q}(X,H)^{rc},H)$ and \newline
%   $(\mathbb{Q}(X',H')^{rc},H')$,
%   \item $X$ is algebraically independent over $G$, and $X'$
%   is algebraically independent over $G'$ with $\beta(X)=X'$,
%   \item $\Gamma \leq H$, $\Gamma \leq H'$.
% \end{enumerate}

\medskip\noindent
Note that by Lemma \ref{mainlemma},  $(\mathbb{Q}(X,H)^{\rc},H)$ and
$(\mathbb{Q}(X',H')^{\rc},H')$ as above become $\lam(P; \Gamma)$-substructures of $(K,G)$ and $(K',G')$ respectively. Moreover the map $\beta$ is a partial elementary map between the ordered fields $K$ and $K'$ (in the language $\lam$).

\begin{lemma} The collection $\mathcal{S}$ is a back-and-forth-system.
\end{lemma}
\begin{proof}
Let $\beta:(\Q(X,H)^{\rc},H)\to(\Q(X',H')^{\rc},H')$ be in
$\mathcal{S}$ and take $a \in K\setminus\Q(X,H)^{\rc}$. By symmetry it is enough to prove
that there is $\tilde\beta \in \mathcal{S}$ such that $\tilde\beta$ extends
$\beta$ and $a\in \operatorname{dom}(\tilde\beta)$.

\smallskip\noindent
\underline{Case 1:} $a\in \pi(G)$.\\
Take $b \in K^{m-1}$ such that $(a,b)\in G$. Since $G\subseteq \mathbb A(K)$ and $\mathbb A(K)$ is $\lam$-$\emptyset$-definable of dimension 1, there is $\lam$-$\emptyset$-definable function $f:K\to K^{m-1}$ such that $b=f(a)$.
Let $q(x)$ be the $\lam(P;\Gamma)$-type consisting of the
$\lam$-type of $a$ over $\mathbb{Q}(X,H)^{\rc}$ and for every $l\in\Z$, $h\in H$ and $s>0$ one of the formulas
\begin{align}\label{pureness}
l(x,f(x)) \oplus h &\in sG,\\
l(x,f(x)) \oplus h &\notin sG,
\label{pureness2}
\end{align}
depending on whether it holds true that $l(a,b) \oplus h \in sG$ or not.  Further let
$q'(x)$ be the type over $\Q(X',H')^{\rc}$ corresponding to $q(x)$ via $\beta$. We want to find an
$a' \in K'$ such that $a'$ realizes $q'(x)$. By saturation of
$(K',G')$, it is enough to show that every finite subset of
$q'(x)$ can be realized in $(K',G')$. By o-minimality of $T$, this reduces to find $a'\in K'$ for every $c,d \in \mathbb{Q}(X,H)^{\rc}$ with $c<a<d$ and finite collection of formulas $\phi_1,\dots,\phi_n$ of the form
\eqref{pureness} or \eqref{pureness2} with $(K,G)\models \bigwedge_{i=1}^n \phi_i(a,b)$ such that
\begin{equation}\label{conditionona}
(K',G') \models \beta(c) < a' < \beta(d) \wedge \bigwedge_{i=1}^n \phi_i(a',f(a')).
\end{equation}
By Lemma~\ref{unionofcosets} and the remark succeeding it
$$D:=\{ g \in G' \ : \ (K',G')\models \bigwedge_{i=1}^n \phi_i(g)\}$$
is a finite union of cosets of $tG'$ in $G'$ for some $t \in \mathbb{N}$ and the representatives of these cosets can be chosen to be in $\Gamma$. Then by Lemmas~\ref{denseinsemi} and \ref{projectionlemma}, the set $\pi\big((D\setminus\Gamma\big)\cap \gr(f))$ is dense in the interior of  $\pi(\mathbb A(K')\cap \gr(f))$. Since $\pi(\mathbb A(K)\cap \gr(f))$ is $\lam(\Gamma)$-definable and $a$ is in it, we can assume that the interval $(c,d)$ is a subset of  $\pi(\mathbb A(K)\cap \gr(f))$. As $\beta$ is a partial $\lam$-elementary map, it follows that the interval $\big(\beta(c),\beta(d)\big)$ is a subset of $\pi(\mathbb A(K')\cap \gr(f))$ and that $\pi((D\setminus\Gamma)\cap \gr(f))\cap \big(\beta(c),\beta(d)\big)$ is a dense subset of $\big(\beta(c),\beta(d)\big)$. Now take any $a' \in \pi\big((D\setminus\Gamma)\cap \gr(f)\big)\cap \big(\beta(c),\beta(d)\big)$. This $a'$ satisfies \eqref{conditionona}.
%  We set
% \begin{equation*}
% L := \mathbb{Q}(X,H,a)^{\rc} \cap G
% \textrm{ and } ':= \mathbb{Q}(X',H',a')^{\rc}\cap G'.
% \end{equation*}
It is clear that $[H\cup \{(a,b)\}]_G$ and $[H'\cup\{(a',f(a'))\}]_{G'}$ are pure subgroups of $G$ and $G'$ respectively. Let $\tilde\beta$ be the $\lam(P;\Gamma)$-isomorphism that extends $\beta$ to $\mathbb{Q}(X,H,a)^{\rc}$ and maps $a$ to $a'$. By conditions
\eqref{pureness} and \eqref{pureness2}, we have that $\tilde\beta$ maps $[H\cup\{(a,b)\}]_G$ onto
$[H'\cup\{(a',f(a'))\}]_{G'}$. Hence $\tilde\beta$ is an $\lam(P;\Gamma)$-isomorphism between
$\big(\mathbb{Q}(X,H,a)^{\rc},[H\cup\{(a,b)\}]_G\big)$ and
$\big(\mathbb{Q}(X',H',a')^{\rc},[H'\cup\{(a',f(a'))\}]_{G'}\big)$ and $\tilde\beta \in \mathcal{S}$.

\medskip\noindent
\underline{Case 2:} $a \in \mathbb{Q}(X,G)^{\rc}$.\\
Let $g_1,\dots,g_n \in G$ such that $a \in \mathbb{Q}\big(X,\{g_1,\dots,g_n\}\big)^{\rc}$.
By applying the previous case $n$ times, we get a $\tilde\beta \in \mathcal{S}$ such that
$g_1,\dots,g_n \in \operatorname{dom}(\tilde\beta)$. Since
$\operatorname{dom}(\tilde\beta)$ is a real closed field, we have $a\in \operatorname{dom}(\tilde\beta)$.

\medskip\noindent
\underline{Case 3:} $a \notin \mathbb{Q}(X,G)^{\rc}$.\\
Let $C$ be the cut of $a$ in $\mathbb{Q}(X,H)^{\rc}$ and let $C'$ be the
 cut in $\mathbb{Q}(X',H')^{\rc}$ corresponding to $C$ via $\beta$. By
saturation, we can assume that there are $c',d' \in K'$ such
that every element in the interval $(c',d')$ realizes the cut $C'$.
Let $u \in K^{|X|}$ be the set $X$ written as a tuple. Let
$f_1,\dots,f_n:K^{mt+|X|}\to K$ be $\emptyset$-definable functions in the language $\lam(\Gamma)$. By Corollary \ref{smallnesscor}, we
know that there exists $b' \in (c',d')$ such that for $i=1,\dots,n$ and
every tuple $g_1',\dots,g_t'$ of elements of $G'$
\begin{equation*}
f_i(g_1',\dots,g_t',\beta(u)) \neq b'.
\end{equation*}
 Thus by saturation, there is an $a' \in (c',d')$ such that $a'
\notin \mathbb{Q}(X',G')^{\rc}$. Since $a'$ realizes the cut $C'$, there is an
$\lam(\Gamma)$-isomorphism $\tilde\beta$ from $\mathbb{Q}(X,a,H)^{\rc}$ to
$\mathbb{Q}(X',a',H')^{\rc}$ extending $\beta$ and sending $a$ to $a'$. Since $a
\notin \mathbb{Q}(X,G)^{\rc}$ and $a' \notin \mathbb{Q}(X',G')^{\rc}$, using Lemma \ref{mainlemma} we get that
\begin{equation*}
(\mathbb{Q}(X,a,H)^{\rc})^m \cap G = H \textrm{ and } (\mathbb{Q}(X',a',H')^{\rc})^m\cap G' =
H'.
\end{equation*}
Since $\beta(H)=H'$ and $\tilde\beta$ extends $\beta$, we get that
$\tilde\beta$ is an $\lam(P;\Gamma)$-isomorphism from $(\mathbb{Q}(X,a,H)^{\rc},H)$ to
$(\mathbb{Q}(X',a',H')^{\rc},H')$ with $\tilde\beta(X\cup \{a\})=X'\cup\{a'\}$. Thus we
have that $\tilde\beta \in \mathcal{S}$.
\end{proof}

\medskip\noindent
Now the proof of completeness becomes an easy consequence of this lemma.

\begin{theorem}\label{completeness}
Let $\Gamma$ be t-dense in $\mathbb A(\R)$.  Then the theory $T(\Gamma)$ is complete.
\end{theorem}

\begin{proof}
Take $|\Gamma|^{+}$-saturated models $(K,G)$ and $(K',G')$ of $T(\Gamma)$, and let $\Cal S$ be as above. It only remains to show that $\Cal S$ is non-empty. But it is easy to see that the identity map on $(\Q(\Gamma)^{\rc},\Gamma)$ belongs to $\Cal S$.
 \end{proof}

\subsection{Quantifier elimination}
Let $x=(x_1,\dots,x_t)$ be a tuple of distinct variables. For every $\lam(P;\Gamma)$-formula $\phi(x)$ of the form
\begin{equation}\label{formform}
\exists y_1 \cdots \exists y_{mn}
\bigwedge_{j=1}^{n} P(y_{m(j-1)+1},\dots,y_{mj}) \wedge \psi(x,y_1,\dots,y_{mn}),
\end{equation}
where $\psi(x,y)$ is an $\lam(\Gamma)$-formula, let $P_{\phi}$ be a new relation symbol of arity $t$, and let $\lam(P;\Gamma)^+$ be the
language $\lam(P;\Gamma)$ together with relation symbols $P_{\phi}$ (for various $x$).

\medskip\noindent
Let $T(\Gamma)^+$ is the $\lam(P;\Gamma)^+$-theory
extending the theory $T(\Gamma)$ by an axiom:
\begin{equation*}
\forall x\big(P_{\phi}(x) \leftrightarrow \phi(x)\big),
\end{equation*}
for each $\phi$ of the form (\ref{formform}).

\medskip\noindent
With this notation in hand we are ready to prove the promised quantifier elimination result.

\begin{theorem}\label{nearMC2}
The theory $T(\Gamma)^+$ has quantifier elimination.
\end{theorem}
\begin{proof} Let $(K,G)$ and $(K',G')$ be two $|\Gamma|^{+}$-saturated models of $T(\Gamma)^{+}$ and let  $\mathcal{S}$ be the back-and-forth system between $(K,G)$ and $(K',G')$ constructed above.
Also take $a=(a_1,\dots,a_n) \in K^n$
and $b=(b_1,\dots,b_n) \in (K')^n$ with the same
quantifier-free $\lam(P;\Gamma)^+$-type.  In order to prove quantifier elimination, we just need to find $\tilde\beta \in \mathcal{S}$ sending $a$ to $b$.
Without loss of generality, we may assume
that $\{a_1,\dots,a_r\}$ is a transcendence basis of $\Q(G,a)$ over $\Q(G)$. Since $a$ and $b$ have the same quantifier-free $\lam(P;\Gamma)^+$-type, we get that
$\{b_1,\dots,b_r\}$ is a transcendence basis of $\Q(G',b)$ over $\Q(G')$.  Let $\beta$ be the $\lam(\Gamma)$-isomorphism between
$\mathbb{Q}(a_1,\dots,a_r,\Gamma)^{\rc}$ and $\mathbb{Q}(b_1,\dots,b_r,\Gamma)^{\rc}$ sending $a$ to $b$. We will now show that $\beta$ extends to an isomorphism $\tilde\beta$ in the back-and-forth-system $\mathcal{S}$. Let $g_1,\dots,g_t \in G$ be such that $a_{r+1},\dots,a_n$ are in $\mathbb{Q}(a_1,\dots,a_r,g_1,\dots,g_t,\Gamma)^{\rc}$.
Let $q(x_1,\dots,x_t)$ be the $\lam(P;\Gamma)$-type consisting of the
$\lam(\Gamma)$-type of $(g_1,\dots,g_t)$ over $\mathbb{Q}(a_1,\dots,a_r)^{\rc}$ and for every $k_1,\dots,k_t\in\Z$, $s\in\N$ and $\gamma\in\Gamma$ one of the formulas
\begin{align}\label{qepureness}
\bigwedge_{i=1}^t P(x_i)\wedge\bigoplus_{i=1}^{t} k_i x_i \oplus \gamma &\in sG, \text{ or}\\
\bigwedge_{i=1}^t P(x_i)\wedge\bigoplus_{i=1}^{t} k_i x_i \oplus \gamma &\notin sG,
\label{qepureness2}
\end{align}

\noindent
depending on whether $\bigoplus_{i=1}^{t} k_i g_i \oplus \gamma \in sG$. Let
$q'$ be the type corresponding to $q$ under $\beta$. We want to find
$h_1,\dots,h_t \in G'$ realizing $q'$. By saturation of
$(K',G')$, it is enough to show that every finite subset of
$q'$ can be realized. So let $\psi(x,b_1,\dots,b_r)$ be an $\lam(\Gamma)$-formula  in $q'$ and $\chi_1(x),\dots,\chi_e(x)$ be finitely many formulas in $q'$ of  the form \eqref{qepureness} or \eqref{qepureness2}. Put $\chi=\bigwedge_{i=1}^{e}\chi_i$. By Lemma \ref{unionofcosets}, the set
$$\big\{ (h_1,\dots,h_t) \in G'^{t} \ : \ (K',G')\models\chi(h_1,\dots,h_t))\big\}$$
is a finite union of cosets of $(lG')^t$ in $(G')^t$ for some $l \in \mathbb{N}$. So the formula $\chi(x)$ is equivalent to an $\lam(P;\Gamma)$-formula of the form \eqref{formform}. Hence the disjunction $\psi\wedge \chi$
% is also of this form. Hence
% \begin{equation*}
% \exists y_1 \cdots \exists y_t \bigwedge_{i=1}^{t} P(y_i) \wedge \psi(y_1,\dots,y_t,b_1,\dots,b_r) \wedge \chi(y_1,\dots,y_t,b_1,\dots,b_r)
% \end{equation*}
is a quantifier-free $\lam(P;\Gamma)^{+}$-formula. Since $(a_1,\dots,a_r)$ and $(b_1,\dots,b_r)$ have the same quantifier-free $\lam(P;\Gamma)^{+}$-type, and
$$\exists x(\psi\wedge\chi)(x,a_1,\dots,a_r)$$
holds in $(K,G)$, the corresponding formula $\exists x(\psi\wedge\chi)(x,b_1,\dots,b_r)$ holds in $(K',G')$.
% \begin{equation*}
% \exists y_1 \cdots \exists y_t \bigwedge_{i=1}^{t} P(y_i) \wedge \psi(y_1,\dots,y_t,b_1,\dots,b_r) \wedge \chi(y_1,\dots,y_l,b_1,\dots,b_r)
% \end{equation*}
% holds in $(K',G')$.
So $q'$ is finitely satisfiable. Now let $h_1,\dots,h_t \in G'$ realize $q'$. Then $\beta$ extends to a field isomorphism
$$\tilde\beta:\mathbb{Q}\big(a_1,\dots,a_r,g_1,\dots,g_t,\Gamma\big)^{\rc}\to\mathbb{Q}\big(b_1,\dots,b_r,h_1,\dots,h_t,\Gamma\big)^{\rc}.$$
By the construction of $g_1,\dots,g_t$ and $h_1,\dots,h_t$, we have that
\begin{equation*}
\bigoplus_{i=1}^{t} k_i g_i \oplus \gamma \in sG \textrm{ if and only if } \bigoplus_{i=1}^{t} k_i h_i \oplus \beta(\gamma) \in sG'
\end{equation*}
for all $k_1,\dots,k_t \in \mathbb{Z}$, $s \in \mathbb{N}$ and $\gamma \in \Gamma$. Hence $\tilde\beta$ is in $\mathcal S$.
% an $\lam(P;\Gamma)$-isomorphism of  $$\big(\mathbb{Q}(a_1,\dots,a_r,g_1,\dots,g_t)^{\rc},[\Gamma\cup\{g_1,\dots,g_t\}]_G\big)\text{ and }$$ $$\big(\mathbb{Q}(b_1,\dots,b_r,h_1,\dots,h_t)^{\rc},[\Gamma\cup\{h_1,\dots,h_t\}]_{G'}\big).$$
% It is also easy to see that $\tilde\beta \in \mathcal{S}$.

\end{proof}

\subsection{Induced structure}
Let $(K,G)$ be a model of $T(\Gamma)$. Here we study the subsets of $G^n$ definable in $(K,G)$.

\medskip\noindent
Let $B$ be a set of parameters such that $B\setminus \pi(G)$ is algebraically independent over $\Q(G)$.

\begin{proposition}\label{induced}
Let $X \subseteq G^n$ be definable in $\big(K,G,(\gamma)_{\gamma\in\Gamma}\big)$ with parameters from $B$. Then $X$ is a finite union of sets of the form
\begin{equation}\label{semialg&cosets}
E \cap(\gamma\oplus (sG)^n),
\end{equation}
where $E$ is $\lam(\Gamma)$-$B$-definable, $\gamma\in \Gamma^n$, and $s \in \mathbb{N}$.
\end{proposition}

\begin{proof} We may assume that $(K,G)$ is a $|\Gamma|^{+}$-saturated model of $T(\Gamma)$.
Let $\mathcal{S}$ be the back and forth
system of $\lam(P;\Gamma)$-isomorphisms between $(K,G)$ and itself constructed above. Let $g,h\in G^n$ such that
for every $E\subseteq K^{mn}$ definable in $\big(K,(\gamma)_{\gamma\in\Gamma}\big)$ over $B$, $\gamma_1,\dots,\gamma_t\in\Gamma^n$, and $s\in\N$ we have that
\begin{equation}\label{semialg&cosets2}
g\in E \cap (\gamma \oplus (sG)^n)\Leftrightarrow h\in E \cap (\gamma \oplus (sG)^n).
\end{equation}

%
%
% Let $g\in G^n$ and
% let $p(x)$ be the $\lam(P;\Gamma)$-type consisting of all formulas of the form
% \begin{equation*}\label{eqinduced1}
%  \varphi(x) \wedge
% \bigvee_{i=1}^{l} x \in \gamma_i \oplus (sG)^n,
% \end{equation*}
% % where $\varphi$ is an $\lam(\Gamma)$-formula with parameters from $B$, $\gamma_1,\dots,\gamma_l \in \Gamma^n$, $k \in \mathbb{Z}^n$ and $s \in \mathbb{N}$.
% Let $h \in G^n$ be such that $h$ satisfies $p$.
Note that by Lemma \ref{unionofcosets}, the collection of finite unions of sets of the form \eqref{semialg&cosets} is closed under boolean operations.
Hence it suffices to show that there is
$\beta\in \mathcal{S}$ fixing $B$ and mapping $g$ to
$h$. Since $h$ satisfies all $\lam(\Gamma)$-formulas over $B$ which are satisfied by $g$, there is an $\lam(\Gamma)$-isomorphism from $\mathbb{Q}(g,B)^{\rc}$ to $\mathbb{Q}(h,B)^{\rc}$ fixing $B$ and mapping $g$ to $h$. To show that $\beta \in \mathcal{S}$, it is only left to prove that $\beta$ takes $G\cap(\Q(B,g)^{\rc})^m$ to $G\cap(\Q(B,h)^{\rc})^m$. Using Lemma \ref{mainlemma} it suffices to show
$$\beta([\Gamma\cup((\Q(B)^{\rc})^m\cap G)\cup\{g\}]) = [\Gamma\cup((\Q(B)^{\rc})^m\cap G)\cup\{h\}].$$
It is enough to show for all $\gamma \in \Gamma^n$, $k \in \mathbb{Z}^n$ and $s \in \mathbb{N}$ that
\begin{equation*}
g \in \gamma \oplus D_{k,s} \textrm{ if and only if } h \in \gamma \oplus D_{k,s},
\end{equation*}
since we can choose representatives for cosets of $D_{k,s}$ in $G^n$ from $\Gamma^n$.
By Lemma \ref{unionofcosets}, there are $\gamma_1,\dots,\gamma_{t_1},\delta_1,\dots,\delta_{t_2} \in \Gamma^n$ such that $\gamma \oplus D_{k,s}= \bigcup_{i=1}^{t_1} \gamma_i \oplus (sG)^n$ and $G^n\setminus(\gamma \oplus D_{k,s}) = \bigcup_{i=1}^{t_2} \delta_i \oplus (sG)^n$. We are done since by assumption $g \in \gamma \oplus D_{k,s}$ if and only if $h \in \gamma \oplus D_{k,s}$.
\end{proof}

\begin{corollary}\label{induced-corollary}Let $X \subseteq G^n$ be definable in $\big(K,G,(\gamma)_{\gamma\in\Gamma}\big)$ with parameters from $B$. Then the topological closure $\overline{X}$ of $X$ is definable in $\big(K,(\gamma)_{\gamma\in\Gamma}\big)$ over $B$.
\end{corollary}

\begin{proof} By Lemma 1.3.4 of \cite{vdDries-book} the topological closure of a set definable in the ordered field $K$ is again definable in the ordered field $K$. So it suffices to prove that there is an $\lam(\Gamma)$-$B$-definable set $E\subseteq K^{mn}$  such that $X$ is a dense subset of $E$. We do this by induction on $n$. For $n=0$, the case is trivial. So let $n>0$. By Proposition \ref{induced} we may assume that there exist an $\lam(\Gamma)$-$B$-definable set $E_0$ and an $\lam(P;\Gamma)$-$\emptyset$-definable set $D_0\subseteq G^n$ such that $X=E_0 \cap D_0$. By Lemma \ref{denseinsemi}, we can assume that $D_0$ is t-dense in $\mathbb A(K)^n$. Since the t-topology and the induced topology from $K^m$ coincide apart from finitely many points, we can even assume that $D_0$ is dense in a $\lam(\Gamma)$-$\emptyset$-definable $S_0$.
By cell decomposition, we can assume that $E_0$ is a cell and that $E_0\subseteq S_0 \cap \mathbb{A}(K)^n$. Hence $\dim E_0 \leq n$.
First consider the case that $\dim E_0 = n$. By Lemma 4.1.15 of \cite{vdDries-book}, we can assume that $E_0$ is open in $S_0$. Since $D_0$ is dense in $S_0$, $X$ is dense in $E_0$. Now consider the case that $\dim E_0 =s$ for $s<n$. We can assume that there is an $\lam(\Gamma)$-$B$-definable set $C\subseteq \mathbb{A}(K)^s$, a projection $\pi : K^{mn} \to K^{ms}$ and an $\lam(\Gamma)$-$B$-definable continuous function $f$ such that $\pi(E_0)=C$ and $f(C)=E_0$.
Consider the set
\begin{align*}
V':= \{ (g_1,\dots,g_{s}) \in G^s \cap C \ : \ f(g_1,\dots,g_{s}) \in D_0\}.
\end{align*}
By the induction hypothesis, there is an $\lam(\Gamma)$-$B$-definable subset $E_1$ of $C$  such that $V'$ is dense in $E_1$. By continuity of the $f$, the image of $V'$ under $f$ is dense in the image of $E_1$ under $f$. Set $E:=f(E_1)$. Since $X=f(V')$, we have that $X$ is dense $E$.

\end{proof}

\section{Elliptic curves} \label{ellipticcurve}
\noindent Fix $a,b\in\Q$ such that $4a^3+27b^2\neq 0$ and let $\Delta$ be the subset of $\mathbb{R}^2$ defined by
\begin{equation*}
\{ (x,y) \in \mathbb{R}^2 \ : \ y^2=x^3+ax+b \}.
\end{equation*}
Further let $(c,d) \in \mathbb{Q}^2\setminus\Delta$ and define
\begin{equation*}
\Delta^* := \Delta \cup \{(c,d)\}\text{ and }\Delta^*(\Q):=\Delta\cap\Q^2.
\end{equation*}
In this section, we show that $\Delta^*$ can be given the structure of a definable group in $\mathbb{R}$ such that $\Delta^*(\mathbb{Q})$ becomes a subgroup with the Mordell-Lang property.\\

\noindent Let $\mathbb{P}^2(\C)$ denote the complex projective plane and we write its elements as $(\alpha:\beta:\gamma)$. Let $\Cal E$ consist of $(\alpha:\beta:\gamma)\in\mathbb P^2(\C)$ such that
$$\beta^2\gamma=\alpha^3+a\alpha\gamma^2+b\gamma^3.$$
Then $\Cal E$ is an \emph{elliptic curve} and it is well-known that it becomes a group with a group operation $\oplus$ given by rational functions over $\Q$ and identity element $\Cal O:=(0:1:0)$ (see for instance III.2.3 in \cite{silverman}). Now $\Delta^*$ can be mapped injectively into $\Cal E$ by
\begin{align*}
\iota: \Delta^* &\to \Cal E\\
(x,y) &\mapsto \left\{
                 \begin{array}{ll}
                   \Cal O, & \hbox{if $(x,y)=(c,d)$;} \\
                   (x:y:1), & \hbox{otherwise.}
                 \end{array}
               \right.
\end{align*}
We write $\Cal E(\R)$ for the image of $\Delta^*$ under the map $\iota$ and $\Cal E(\Q)$ for the image of $\Delta^*(\Q)$. \\

\noindent It is easy to see that both $\Cal E(\R)$ and $\Cal E(\Q)$ are subgroups of $\Cal E$. Thus the map $\iota$ induces a group structure on $\Delta^*$ and $\Delta^*(\Q)$ becomes a subgroup of $\Delta^*$. Since the group structure on $\Cal E$ is given by rational functions, the group structure on $\Delta^*$ is semialgebraic.
% \sout{Further since $\Cal E$ is compact, $\Cal E(\Q)$ is dense in $\Cal E(\R)$ whenever $\Delta^*(\Q)$ is infinite. Hence $\Delta^*(\Q)$ is dense in $\Delta^*$ if it is infinite.}
%
% \medskip\noindent
% \sout{It is just left to show that $\Delta^*(\Q)$ has the Mordell-Lang property.}

\medskip\noindent
The elliptic curve $\Cal E$ is a complex Lie group and as such it is isomorphic to a complex torus $\C/\Lambda$ where $\Lambda\subseteq \C$ is a lattice.
%; which can be taken to be $\Z+\Z\theta\sqrt{-1}$ for some $\theta\in\R$ since $\Cal E$ is defined over $\Q$ (see VI.5.5 and VI.6.7(d) in \cite{silverman}).
This isomorphism uses the Weierstrass elliptic function $\wp$ attached to $\Lambda$, namely
\begin{align}\quad\quad\quad\quad f:\C/\Gamma\to& \quad\Cal E\nonumber\\
z+\Lambda\mapsto&\left\{\begin{array}{ll} (\wp(z):\wp'(z):1) & \mbox{if $z\notin\Lambda$,}\\
                                  \Cal O  & \mbox{otherwise.}
                \end{array}
         \right .\nonumber
\end{align}

\noindent
We also have the quotient map $q:\C\to\C/\Gamma$.

\noindent The endomorphism ring of $\Cal E$ is either $\Z$ or $\Z[\tau]$ for some $\tau\in\C$ with $\tau^2$ is a negative integer. In the second case, we say $\mathcal E$ has \emph{complex multiplication by $\tau$}.

\medskip\noindent
In the case that $\mathcal E$ does not have complex multiplication, all algebraic subgroups of $\mathcal E^n$ are given as the kernels of maps of the form
\begin{equation}\label{linear}
 (x_1,\dots,x_n)\mapsto k_1x_1\oplus\dots\oplus k_nx_n: \Cal E^n\to \Cal E,
\end{equation}
where $k_i \in \mathbb{Z}$ for $i=1,\dots,n$.

\medskip \noindent
If the elliptic curve $\mathcal E$ has complex multiplication by $\tau$, then the situation is a bit more complicated. Because an algebraic subgroup of $\mathcal E^n$ is the kernel of a map of the form
\begin{equation}\label{cmsub}
(x_1,\dots,x_n)\mapsto \bigoplus_{i=1}^{n} (k_i + l_i \tau) x_i: \Cal E^n\to \Cal E,
\end{equation}
with $k_i,l_i \in \mathbb Z$. However, using the following lemma we still have control over the intersection of these subgroups with $\Cal E(\Q)^n$.

\begin{lemma}\label{intfinite} Let $\Cal E$ be an elliptic curve with complex multiplication by $\tau$. Then the intersection of $\mathcal E(\mathbb R)$ with its image under $\tau$ is finite.
\end{lemma}
\begin{proof}
In this case $\Cal E$ is isomorphic to $\C/\Lambda$, where $\Lambda=\Z+\Z\tau$ (see C.11.6 in \cite{silverman}).
% By \cite{silverman} VI.5.5 there are $a,b \in \mathbb{Z}$ such that the sublattice $L$ of $\mathbb C$ generated by $1$ and $a+b\tau$ is isomorphic to $\mathcal E$ via $(\wp,\wp'): \mathbb C / L \to \mathcal{E}$. Since $\mathcal E$ is defined over $\mathbb R$, we can assume that $a=0$ and $b=1$ by \cite{silverman} VI.6.7(a).
Since $\tau$ is purely imaginary, the series expansions of $\wp$ and $\wp'$ have only real coefficients (see Theorem VI.3.5 in \cite{silverman}).
%  \sout{Hence $\wp$ and $\wp'$ map equivalence classes of elements of the real axis to elements of the real axis.}}
Then $f$ maps the set $ [0,1)+\Lambda$ into $\mathcal{E}(\R)$. Let $S$ be  the inverse image of $\mathcal E(\mathbb R)$ under $f$. Being a one dimensional group definable in the ordered field $\R$, $\Delta^*$ has finitely many connected components and so does $\mathcal E(\R)$. Thus $S$ is
% either}
% \begin{equation*}
%  [0,1) + \Lambda\quad\textrm{ or }\quad [0,1)+\frac{s}{2}\cdot \tau + \Lambda .
% \end{equation*}
the image under $q$ of a finite union of lines in $\C$ parallel to the real axis. On $\mathbb C / \Lambda$ the endomorphism of $\mathcal E$ corresponding to $\tau$ is just multiplication by the complex number $\tau$ (This means the map taking $x+\Lambda$ to $\tau x+\Lambda$. See Theorem VI.4.1 in \cite{silverman}). Hence the set $\tau S$ is the image under $q$ of a finite union of lines parallel to the imaginary axis. Therefore the intersection of $S$ and $\tau S$ is finite and so is the intersection of $\mathcal E(\mathbb R)$ and its image under $\tau$.
\end{proof}

\medskip\noindent
The key fact we use is the following special case of Faltings' Theorem (see \cite{faltings}).

\begin{theorem} \label{mordelllangtheorem2} Let $\mathcal E$ be an elliptic curve over $\Q$ and $\Gamma$ a finitely generated subgroup of $\mathcal E$. Then for every algebraic subset $V$ of $\mathcal E^n$, the set $V\cap \Gamma^n$ is a finite union of cosets of subgroups $A\cap \Gamma^n$ of $\Gamma^n$, where $A$ is an algebraic subgroup of $\mathcal E^n$.
\end{theorem}

\noindent
Now we are ready to prove that $\Delta^*(\Q)$ has the Mordell-Lang property.
\begin{proposition}
The group $\Delta^*(\Q)$ has the Mordell-Lang property.
\end{proposition}

\begin{proof}
It is enough to show that for every algebraic subset $V$ of $\Cal E^n$ the set $V\cap\Cal E(\Q)^n$ is a finite union of cosets in $\Cal E(\Q)^n$ of subgroups of $\Cal E(\Q)^n$ given as the kernel of maps of the form
$$(x_1,\dots,x_n)\mapsto k_1x_1\oplus\dots\oplus k_nx_n:\Cal E(\Q)^n\to\Cal E(\Q),$$
with $k_1,\dots,k_n\in\Z$.

\medskip\noindent
By the Mordell-Weil Theorem, $\mathcal E (\mathbb{Q})$ is indeed a finitely generated subgroup of the elliptic curve $\mathcal E$. So in the case that $\Cal E$ does not have complex multiplication, Theorem~\ref{mordelllangtheorem2} gives the desired result directly.

\medskip\noindent
If $\Cal E$ has complex multiplication, say by $\tau$, then Lemma \ref{intfinite} implies that the cosets of the intersection of an algebraic subgroup of $\Cal E^n$ given as the kernel of a map of the form (\ref{cmsub}) and $\mathcal E (\mathbb Q)^n$ is a finite union of cosets of the intersection of $\mathcal E(\mathbb Q)^n$ with subgroups given as kernels of maps of the form (\ref{linear}). Hence the result follows again from Theorem~\ref{mordelllangtheorem2}.

\end{proof}

% \noindent In this section, we have shown that $\Delta^*$ can be given a semialgebraic group structure such that $\Delta^*(\Q)$ forms a Mordell-Lang subgroup of $\Delta^*$. Moreover if $\Delta^*(\Q)$ is infinite, $\Delta^*(\Q)$ is dense in $\Delta^*$.
%In the rest of the paper, we will establish the following Theorem. Let $\lam$ be the language of ordered rings and let $\lam(P)$ be the language $\lam$ together with an $m$-ary predicate $P$.
%
% \begin{theorem}\label{nearMC2}
% Let $(\mathbb{A},\oplus)$ be a 1-dimensional group of $\R^m$ definable in $\R$ and let $\Gamma$ be a dense subgroup of $\mathbb{A}$ with the Mordell-Lang property. Then every subset of $\R^s$ definable in the structure $(\R,\Gamma)$ is defined by a boolean combination of formulas of the form
% $$\exists y_1\cdots\exists y_{mn}\big[\bigwedge_{j=0}^{n-1}P(y_{2j+1},...,y_{2j+m})\wedge\phi(x,y)\big],$$
% where $y$ denotes the tuple $(y_1,\dots,y_{2n})$, $x$ is an $s$-tuple of distinct variables and $\phi(x,y)$ is a quantifier-free $\lam$-formula. Moreover, every open definable set in $(\R,\Gamma)$ is semialgebraic.
% \end{theorem}
%
\begin{proof}[Proof of Theorem \ref{nearMC}]
Note that $\Cal C$ from the introduction is just $\Delta^*(\Q)$ without $(c,d)$ and thus $(\R,\Cal C)$ and $(\R,\Delta^*(\Q))$ are quantifier-free interdefinable over $\emptyset$. Hence it suffices to prove Theorem~\ref{nearMC} with $\Delta^*(\Q)$ in the place of $\mathcal C$.

\medskip\noindent
If $\Delta^*(\Q)$ is finite, then Theorem \ref{nearMC} is trivial. Hence we may assume that $\Delta^*(\Q)$ is infinite. First note that $\Delta$ has either one or two connected components depending on whether
the polynomial $p(X)=X^3+aX+b$ has one or three real roots.
By the construction in \cite{Pillay}, p.247, the t-topology on $\Delta^*$ is given by
\begin{equation*}
 \{ Z \subseteq \Delta^* : g\oplus Z \cap \Delta \textrm { is open, for every } g \in \Delta^*\}.
\end{equation*}
One can easily check that the number of $t$-connected components of $\Delta^*$ coincides with the number of connected components of $\Delta$. Since $\Delta^*$ can be embedded into $\mathbb P^2(\mathbb C)$, $\Delta^*$ is t-compact.
Let $H$ be the t-connected component of $\Delta^*$ containing the identity of the group $\Delta^*$. By \cite{Pillay} Corollary 2.10, $H$ is definable. Since $\Delta^*$ has at most two t-connected components, the index of $H$ in $\Delta^*$ is at most 2. Hence $H\cap \Delta^*(\Q)$ is infinite since $\Delta^*(\Q)$ is infinite. Since $H$ is t-compact, t-connected and 1-dimensional, $H\cap \Delta^*(\Q)$ is t-dense in $H$. Moveover, since $H$ is a subgroup of $\Delta^*$ of finite index, the structures $(\R,H\cap \Delta^*(\Q))$ and $(\R,\Delta^*(\Q))$ are existentially interdefinable over $\emptyset$.
Now $H$ can be taken as $\mathbb A(\R)$ of the previous section and $\Delta^*(\Q)\cap H$ as $\Gamma$ (It is clear that $(\Delta^*(\Q)\cap H)/n(\Delta^*(\Q)\cap H)$ is finite for every $n>0$).
Therefore Theorem \ref{nearMC} follows immediately from Theorem \ref{nearMC2}. \end{proof}
% As will be shown in the next section, Theorem \ref{opencore-elliptic} will directly from Theorem \ref{nearMC} and Corollary \ref{induced-corollary}.
\noindent Also Theorem \ref{completeness} takes the following form in this setting.
\begin{theorem}\label{completeness-elliptic}
Suppose that $\Delta^*(\Q)$ is infinite.  Let $K$ be a real closed field and $G$ be a subgroup of $\Delta^*(K)$ with a group morphism
$$\delta\mapsto\delta':\Delta^*(\Q)\to G,$$
and let $\Delta'$ be the image of $\Delta^*(\Q)$ under this map. Then $$\big(K,G,(\pi(\delta'))_{\delta\in\Delta^*(\Q)}\big)\equiv\big(\R,\Delta^*(\Q),(\pi(\delta))_{\delta\in\Delta^*(\Q)}\big)$$
if and only if
\begin{itemize}
 \item $\big(K,(\pi(\delta'))_{\delta\in\Delta^*(\Q)}\big)\equiv\big(\R,(\pi(\delta))_{\delta\in\Delta^*(\Q)}\big)$,
%\item $\Delta'$ is a pure subgroup of $G$,
\item for every $n>0$ and $g\in G$ we have that $g\in\Delta'$ whenever $ng\in\Delta'$,
\item for every $n>0$, $|G/nG|=|\Delta'/n\Delta'|$,
\item for every t-connected component $Y$ of $\Delta(\R)$,
$$\Delta^*(\Q)\cap Y\text{ is dense in }Y\Leftrightarrow G\cap Y(K) \text{ is dense in }Y(K).$$
\item $G$ satisfies the same Mordell-Lang conditions as $\Delta'$.
\end{itemize}
\end{theorem}

\section{Open core}\label{opencoresection}

% In this section we prove that there are no open sets that are definable in the structures studied in the previous sections that are not definable in the field of reals. This is to say that their {\it open core}, as will be defined below, is nothing more than the field of reals.
% However
\noindent
Here we work in a more general setting than the previous sections: Let $\mathcal{R}=(R,<,+,\dots)$ be an o-minimal expansion of a densely ordered abelian group in a language $\Cal L=\{<,+,\dots\}$ and let $T_{\mathcal{R}}$ be its theory.
We say a subset $D$ of ${R}^m$ is \emph{small}, if there is no $n\in \N$, no interval $I\subseteq R$ and no surjection $f: D^n \to I$ which is definable in $\mathcal R$. Note that for real closed fields this notion of smallness is equivalent to the notion defined in Section 2.  Now take a small subset $G$ of ${R}^m$.
 Let $T_{\mathcal{R}}(G)$ be the theory of $(\mathcal{R},G)$ in the language $\la(P)=\la\cup\{P\}$, where $P$ is a new $m$-ary relation symbol. We denote models of $T_\Cal R(G)$ by $(\Cal M,P)$, where $\Cal M$ is an $\la$-structure. In what follows, we say that a set $B$ is $\dcl$-{\it independent over $P$} if for every $b\in B$, the singleton $\{b\}$ is not definable in $\Cal M$ over $\pi_1(P)\cup\cdots\cup\pi_m(P)\cup B\setminus\{b\}$.
% (Here and in the rest of the section for a subset $A$ of $M$, $A\setminus P$ means $A\setminus (\pi_1(P)\cup\cdots\cup\pi_m(P))$).
(Here and in the rest of the section for a subset $A$ of $M$, $A\setminus P$ means $A\setminus (\pi_1(P)\cup\cdots\cup\pi_m(P))$, and we do not make a distinction between the relation symbol $P$ and its interpretation.) For convenience in some of the proofs we also assume that $\Cal L$ has two constant symbols $c,d$. This way we can combine two $\Cal L$-definable functions by preserving the parameter set as follows: Let $f_1: X_1\to M$ and $f_2:X_2\to M$ be two functions definable in $\Cal M$ over $B$, then the function
$$f:(X_1\times\{c\})\cup(X_2\times\{d\})\to M$$
given by $f(x_1,c)=f_1(x_1)\text{ and }f(x_2,d)=f_2(x_2)$,
is definable in $\Cal M$ over the same parameter set $B$.

\begin{definition}
Let $\Cal A=(A,<,\dots)$ be an ordered structure and let $T'$ be its theory.

\noindent(i) The {\it open core}, denoted by $\Cal A^{\circ}$, of $\Cal A$ is the the structure $\big(A,(U)\big)$, where $U$ ranges over definable open subsets of $A^n$ for various $n>0$.

\smallskip\noindent(ii) We say that a theory $T$ is an {\it open core} of $T'$ if for every $\Cal B'\models T'$, there is $\Cal B\models T$ such that $(\Cal B')^{\circ}$ is interdefinable with $\Cal B$.
\end{definition}

\medskip\noindent
Our main result is the following.

\begin{theorem}\label{opencore}  Suppose that for every $(\mathcal{M},P)\models T_{\mathcal{R}}(G)$ we have that

\begin{itemize}
              \item [(i)]  every subset of $M^s$ definable in $(\Cal M,P)$
 is a boolean combination of subsets of $M^s$ defined by
\begin{equation*}
\exists y_1 \cdots \exists y_{mn}
\bigwedge_{j=0}^{n-1} P(y_{mj+1},\dots,y_{mj+m}) \wedge \phi(x,y_1,\dots,y_{mn}),
\end{equation*}
where $x$ is an $s$-tuple of distinct variables and $\phi$ is an $\Cal{L}$-formula,

 \item [(ii)] for every parameter set $B$ such that $B\setminus P$ is $\dcl$-independent over $P$, and for every set $V \subseteq P^s$ definable in $(\Cal M,P)$ over $B$, its topological closure $\overline{V}\subseteq M^{ms}$ is definable in $\Cal M$ over $B$.
% there is $E\subseteq M^{ms}$ definable in $\Cal M$ over $B$ such that $V$ is contained densely in $E$.
\end{itemize}
Then $T_{\mathcal{R}}$ is an open core of $T_{\mathcal{R}}(G)$.
\end{theorem}

\noindent
On the basis of Theorem \ref{nearMC2} and Corollary \ref{induced-corollary}, this result has the following consequence.

\begin{corollary}\label{opencore-general}
Let $\mathbb A$ and $\Gamma$ be as in Section 3. If $\Gamma$ is t-dense in $\mathbb{A}(\R)$, then $\RCF$ is an open core of $T(\Gamma)$.
\end{corollary}

\noindent
Combining this with the work of the previous section we get  Theorem~\ref{opencore-elliptic} in the following form.

\begin{corollary}
The theory of real closed fields is an open core of the theory of $(\R,\Cal C)$. In particular every open subset of $\R^s$ definable in $(\R,\Cal C)$ is definable in the real field.
\end{corollary}

\medskip\noindent
We prove Theorem \ref{opencore} using the following special case (precisely when $T$ is o-minimal) of Theorem 4.14 from \cite{opencore}.

\begin{theorem}\label{fromBMS} Let $T$ be an o-minimal theory extending the theory of densely ordered abelian groups and let $T'\supseteq T$.
% Then $T$ is an open core of $T'$ if and only if for every $\Cal A\models T'$ every cofinitely continuous unary function definable in $\Cal A$ is given piecewise by function definable in the $\Cal L$-reduct of $\Cal A$.
Then the following are equivalent:
\begin{itemize}
 \item $T$ is an open core of $T'$.
\item For every $\Cal A\models T'$, every unary open set definable in $\Cal A$ is a finite union of intervals and every cofinitely continuous unary function definable in $\Cal A$ is given piecewise by function definable in the reduct of $\Cal A$ to the language of $T$.
\end{itemize}
\end{theorem}

\medskip\noindent
We begin with some technical results. In what follows we assume that every model $(\Cal M,P)$ of $T_{\Cal R}(G)$ satisfies the conditions (i) and (ii) of Theorem \ref{opencore}. Also $B$ always refers to a finite parameter set such that $B\setminus P$ is $\dcl$-independent over $P$.
Note that condition (ii) is equivalent to the following condition:
\begin{enumerate}
 \item[(ii$'$)] for every $V\subseteq P^s$ definable in $(\Cal M,P)$ over $B$, there is $E\subseteq M^{ms}$ definable in $\Cal M$ over $B$ such that $V$ is a dense subset of $E$.
\end{enumerate}

\begin{lemma}\label{denseinimage} Let $X\subseteq M^{mn}$ and $f: X \to M^k$ be definable in $\Cal M$ over $B$ and let $V \subseteq P^n$ be definable in $(\Cal M,P)$ over $B$. Then there is $E\subseteq M^{mn}$ definable in $\Cal M$ over $B$ such that $X \cap V$ is a dense subset of $E$ and $f(X \cap V)$ is dense in $f(E)$.
\end{lemma}
\begin{proof} By o-minimality of $\mathcal{M}$, there are definable subsets $X_1,...,X_l$ of $X$ such that $X=\bigcup_{i=1}^{l} X_i$ and $f$ is continuous on $X_i$ for $i=1,...,l$. Hence we can assume that $f$ is continuous on $X$. By (ii$'$), there is an $\la$-definable set $E$ such that $X \cap V$ is dense in $E$.
Since $f$ is continuous on $X$, the image of $X\cap V$ is dense in the image of $E$. \end{proof}

\begin{lemma}\label{ldefinable} Let $X\subseteq M^{mn}$ and $f_1,f_2: X \to M$ be definable in $\Cal M$ over $B$ and let $D \subset X \cap P^n$ be definable in $(\Cal M,P)$ over $B$. Then the set \begin{equation*}
\bigcup_{d \in D} \{ a\in M : f_1(d) < a < f_2(d)\}
\end{equation*}
is definable in $\Cal M$ over $B$.
\end{lemma}
 \begin{proof} Let $f: X \to M^2$ be the function given by $x\mapsto (f_1(x),f_2(x))$ for $x \in X$. By Lemma \ref{denseinimage}, there is an $\la$-$B$-definable set $E$ such that $D$ is dense in $E$ and $f(D)$ is dense in $f(E)$. Hence
\begin{equation*}
\bigcup_{d \in D} \{ a : f_1(d) < a < f_2(d)\} = \bigcup_{e \in E}  \{ a : f_1(e) < a < f_2(e)\}.
\end{equation*}
Now note that the right hand side of the equation is $\la$-$B$-definable. \end{proof}

\begin{lemma}\label{lboolcomb} Let $X \subseteq M$ be $\mathcal{L}(P)$-$B$-definable. Then $X$ is a finite intersection of sets of the form
\begin{equation*}
f_1(D_1) \cup (M\setminus f_2(D_{2})) \cup Y,
\end{equation*}
where for $i=1,2$, $f_i:M^{mn_i}\to M$ are $\la$-$B$-definable functions, $D_i\subseteq P^{n_i}$ are $\la(P)$-$B$-definable and $Y\subseteq M$ is $\la$-$B$-definable.
% Moreover, $X$ is finite a conjunction of sets of the form
% \begin{equation*}
% g_1(E_1) \cap (M\setminus g_2(E_{2})) \cap Z,
% \end{equation*}
% where $g_1,g_2$ are $\la$-$B$-definable functions, $E_1,E_2$ are $\la(P)$-$B$-definable subsets of $P^n$ and $Z$ is an $\la$-$B$-definable subset of $M$.
\end{lemma}
\begin{proof}  By condition (i) of Theorem \ref{opencore}, $X$ is a boolean combination of sets of the form
\begin{equation*}
\bigcup_{u \in P^n} \{ a \in M \ : \ \mathcal{M} \models \phi(a,u)\},
\end{equation*}
where $\phi$ is an $\la$-$B$-formula. By cell decomposition in $\mathcal M$ applied to $\phi$, we may assume that $X$ is a boolean combination of sets of the form
\begin{align}
\{a\in M: h_1(u) < & a < h_2(u)\text{ for some } u \in C\}\label{form2}\\
\{a\in M: f(u) =& a\text{ for some } u\in D\} \label{form1}
\end{align}
where $f,h_1,h_2:M^{mn}\to M$ are $\la$-$B$-definable functions and $C,D$ are $\la(P)$-$B$-definable subset of $P^n$. After writing $X$ in conjunctive normal form we get
$X=\bigcap X_i$, where $X_i$ is a finite union of sets of the form \eqref{form2} and \eqref{form1} and of sets whose complements are of of the form \eqref{form2} and \eqref{form1}. Using the observation at the end of the first paragraph of this section, we may even assume that $X_i$ is of the form
\begin{equation*}
f_1(D_1) \cup (M\setminus f_2(D_2)) \cup \bigcup_j Y_{j} \cup (M\setminus Z_{j}),
\end{equation*}
where $Y_{j},Z_{j}$ are of the form \eqref{form2}, $f_1,f_2$ are $\la$-$B$-definable functions and $D_1,D_2$ are $\la(P)$-$B$-definable subsets of $P^n$.
By Lemma \ref{ldefinable}, the set $\bigcup_j Y_{j} \cup (M\setminus Z_{j})$ is $\la$-$B$-definable. Thus each $X_i$ is of the form
\begin{equation*}
f_1(D_1) \cup (M\setminus f_2(D_2)) \cup Y,
\end{equation*}
where $Y$ is an $\la$-$B$-definable set.
\end{proof}

\medskip\noindent
{\bf Remark.} By applying this lemma to the complement of $X$ we get that $X$ is a finite union of sets of the form
\begin{equation*}
g_1(E_1) \cap (M\setminus g_2(E_{2})) \cap Z,
\end{equation*}
where for $i=1,2$, $g_i:M^{mn_i}\to M$ are $\la$-$B$-definable functions, $E_i\subseteq P^{n_i}$ are $\la(P)$-$B$-definable and $Z\subseteq M$ is $\la$-$B$-definable.

\begin{proposition}\label{unaryopen} Every unary open set definable set in $(\mathcal{M},P)$ is definable in $\Cal M$.
\end{proposition}

\begin{proof}  Let $X$ be an open subset of $M$ definable in $(\Cal M,P)$. By Lemma \ref{lboolcomb}, there are sets $X_1,\dots,X_l\subseteq M$ such that $X=\bigcap_{i=1}^{l}X_i$ and every $X_i$ is of the form
\begin{equation}\label{thmform}
f_1(D_1) \cup (M\setminus f_2(D_2)) \cup Y,
\end{equation}
where $f_1,f_2$ are $\la$-definable functions, $D_1,D_2$ are $\la(P)$-definable subsets of $P^n$ and $Y$ is an $\la$-definable subset of $M$. Since $X$ is open,
\begin{equation*}
X = \bigcap_{i=1}^{l} \overset{\circ}{X_i},
\end{equation*}
where $\overset{\circ}{X_i}$ is the interior of $X_i$. Hence it is only left to show that $\overset{\circ}{X_i}$ is definable in $\Cal M$. We will show that $X_i$ is the union of a set $A$ with empty interior and a set $B$ definable in $\mathcal M$. Since $B$ is a finite union of intervals and points, $A$ contributes to the frontier of $X_i$ only. Hence the interior of $X_i$ is the interior $B$.
\\
\noindent Therefore, let $X_i$ be of the form \eqref{thmform}. Consider the
set \begin{equation*}
D:=\{ u \in D_2 \ : \ f_2(u) \notin Y \cup f_1(D_1)\}.
\end{equation*}
By Lemma \ref{denseinimage}, there is an $\la$-definable set $E$ such that $D$
is dense in $E$ and $f_2(D)$ is dense in $f_2(E)$. Hence we get by equation \eqref{thmform} that
\begin{align*}
X_i &= f_1(D_1) \cup (M\setminus f_2(D_2))\cup Y\\
 &= f_1(D_1) \cup (f_2(E)\setminus f_2(D)) \cup (M\setminus f_2(E)) \cup Y.
\end{align*}
By smallness of $P$, $f_1(D_1)$ has empty interior. Since $f_2(D)$ is dense in $f_2(E)$, the set $f_2(E)\setminus f_2(D)$ has empty interior as well. Since $f_1(D_1)\cap f_2(D)$ is empty, the intersection $f_2(E)\cap f_1(D_1)$ is a subset of $f_2(E)\setminus f_2(D)$. Hence  $f_1(D_1) \cup (f_2(E)\setminus f_2(D))$ has empty interior.
Thus $X_i$ is a union of a set with empty interior and a set definable in $\mathcal M$.\end{proof}

\begin{proposition}\label{unaryfunction} Every cofinitely continuous  $\la(P)$-definable unary function is piecewise $\la$-definable.
\end{proposition}
\begin{proof} We may assume that $(\Cal M,P)$ is $|\Cal L|^+$-saturated.
Let  $f: M \to M$ be a cofinitely continuous function which is $\la(P)$-$B$-definable. Take $a\in M$ such that $(\{a\}\cup B)\setminus P$ is $\dcl$-independent over $P$. Note that by saturation and smallness of $P$, such elements form a dense subset $W$ of $M$.\\

\noindent First, we will show that $f(a)$ is $\la$-$B$-definable over $a$.
By Lemma \ref{lboolcomb}, the singleton set $\{ f(a) \}$ is a finite union of sets of the form
\begin{equation*}
f_1(D_1) \cap (M\setminus f_2(D_{2})) \cap Y,
\end{equation*}
where $f_1,f_2$ are $\la$-$(B\cup \{a\})$-definable functions, $D_1,D_2$ are $\la(P)$-$(B\cup \{a\})$-definable subsets of $P^n$ and $Y$ is an $\la$-$(B\cup \{a\})$-definable subset of $M$. Define $D$ to be the set
\begin{equation*}
\{ u \in D_1 \ : \ f_1(u) \in Y \setminus f_2(D_2)\}.
\end{equation*}
Hence
\begin{equation*}
f_1(D)= f_1(D_1) \cap (M\setminus f_2(D_{2})) \cap Y.
\end{equation*}
By Lemma \ref{denseinimage}, there is an $\la$-$(B\cup\{a\})$-definable set $E$ such that $D$ is dense in $E$ and $f_1(D)$ is dense in $f_1(E)$. First note that $f(a)$ must be in the $\la$-$(B\cup\{a\})$-definable set $f_1(E)$. If there is no open interval around $f(a)$ in $f_1(E)$, then $f(a)$ is $\la$-$B$-definable over $a$. So assume for a contradiction, that there is an open interval $I$ around $f(a)$ in $f_1(E)$. Then $f_1(D)\cap I$ is dense in $I$. Hence there are infinitely many $b \in I$ such that $b\in f_1(D)$. This is a contradiction, since $f_1(D)=\{f(a)\}$. Hence $f(a)$ is $\la$-$B$-definable over $a$.

\noindent By the compactness theorem, we get finitely many $\la$-$B$-definable functions $h_1,\dots,h_s$ such that for every $a \in W$, there is $i\in \{1,\dots,s\}$ with $f(a)=h_i(a)$. Let $Z_0$ be the finite set of points of discontinuity of $f$. By monotonicity theorem of o-minimal structures, there is a finite set $Z_1$ such that for every $c,d \in Z_1$ with $(c,d)\cap Z_1=\emptyset$ and for $i,j \in \{1,\dots,s\}$, $h_i,h_j$ are monotone on $(c,d)$, and either $h_i$ and $h_j$ are equal on $(c,d)$ or
\begin{itemize}
\item $h_i(x)< h_j(x)$ for every $x \in (c,d)$ or
\item $h_i(x)> h_j(x)$ for every $x \in (c,d)$.
\end{itemize}
Hence for $c,d \in Z_0 \cup Z_1$ with $(c,d)\cap (Z_0\cup Z_1)=\emptyset$, $f$ is continuous on $(c,d)$. Further $W\cap(c,d)$ is dense in $(c,d)$ and for every $w\in W\cap(c,d)$ we have $f(w)=h_i(w)$ for some $i\in \{1,\dots,s\}$. Since $f$ is continuous on $(c,d)$ and all the $h_i$'s are a positive distance apart, it follows from o-minimality that there is $i\in \{1,\dots,s\}$ such that $f$ and $h_i$ are equal on a dense subset of $(c,d)$ and hence equal on $(c,d)$. So $f$ is given piecewise by $\la$-$B$-definable functions.

\end{proof}

\medskip\noindent
Now Theorem \ref{opencore} follows directly from Theorem \ref{fromBMS} in combination with Propositions \ref{unaryopen}, \ref{unaryfunction}.

\medskip\noindent
{\bf Remark. }As a consequence of Theorem \ref{opencore}, we take care of an unfinished business from \cite{dependent}. Namely
we can simplify Theorem 1.3 of that paper by removing the assumption (iii), since it follows from the first two assumptions using Theorem \ref{opencore}.

\subsection{Subgroups of the unit circle}
Let $\Gamma$ be a subgroup of the unit circle $\mathbb S$ of finite rank. Here we illustrate how Theorem~\ref{opencore-circle} follows from Corollary~\ref{opencore-general}. Note that if $\Gamma$ is finite, then Theorem~\ref{opencore-circle} is trivial, so we assume that it is infinite. Then $\Gamma$ is dense in $\mathbb S$ and $\Gamma$ has the Mordell-Lang property by \cite[Theorem 2]{laurent_exp-poly}.

% \medskip\noindent
% Note that $\Gamma$ is a subgroup of $\C^\times$. Then according to \cite{vdP-Schlickewei} it has the property that for every $a_1,\dots,a_n\in\Q$, there are only finitely many $(\gamma_1,\dots,\gamma_n)\in\Gamma^n$ such that $a_1\gamma_1+\dots+a_n\gamma_n=1$ and $\sum_{i\in I}a_i\gamma_i\neq 0$ for every nonempty subset $I$ of $\{1,\dots,n\}$. Then using Proposition 5.8 of \cite{vdDG} we get that for every algebraic set $V\subseteq \C^n$ the set $V\cap\Gamma^n$ is a finite union of cosets in $\Gamma^n$ of subgroups $T_{k}\cap\Gamma^n$ for various $k\in\Z^n$. Thus it follows that $\Gamma$ has the Mordell-Lang property.

\bibliography{ref}

\begin{thebibliography}{10}

\bibitem{BZ}
Oleg Belegradek and Boris Zilber.
\newblock The model theory of the field of reals with a subgroup of the unit
  circle.
\newblock {\em J. Lond. Math. Soc. (2)}, 78(3):563--579, 2008.

\bibitem{BEG}
Alexander Berenstein, Clifton Ealy, and Ayhan G{\"u}nayd{\i}n.
\newblock Thorn independence in the field of real numbers with a small
  multiplicative group.
\newblock {\em Ann. Pure Appl. Logic}, 150(1-3):1--18, 2007.

\bibitem{opencore}
Alfred Dolich, Chris Miller, and Charles Steinhorn.
\newblock Structures having o-minimal open core.
\newblock {\em Trans. Amer. Math. Soc.}, 362(3):1371--1411, 2010.

\bibitem{vdDries-book}
Lou van~den Dries.
\newblock {\em Tame topology and o-minimal structures}, volume 248 of {\em
  London Mathematical Society Lecture Note Series}.
\newblock Cambridge University Press, Cambridge, 1998.

\bibitem{vdDG}
Lou van~den Dries and Ayhan G{\"u}nayd{\i}n.
\newblock The fields of real and complex numbers with a small multiplicative
  group.
\newblock {\em Proc. London Math. Soc. (3)}, 93(1):43--81, 2006.

\bibitem{Edmundo-Otero}
M{\'a}rio~J. Edmundo and Margarita Otero.
\newblock Definably compact abelian groups.
\newblock {\em J. Math. Log.}, 4(2):163--180, 2004.

\bibitem{faltings}
Gerd Faltings.
\newblock The general case of {S}. {L}ang's conjecture.
\newblock In {\em Barsotti {S}ymposium in {A}lgebraic {G}eometry ({A}bano
  {T}erme, 1991)}, volume~15 of {\em Perspect. Math.}, pages 175--182. Academic
  Press, San Diego, CA, 1994.

\bibitem{thesis_ayhan}
Ayhan G\"{u}nayd{\i}n.
\newblock {\em Model Theory of Fields with Multiplicative Groups}.
\newblock PhD thesis, University of Illinois at Urbana-Champaign, 2008.

\bibitem{dependent}
Ayhan G\"{u}nayd{\i}n and Philipp Hieronymi.
\newblock Dependent pairs.
\newblock Submitted.

\bibitem{NTIII}
Serge Lang.
\newblock {\em Number theory. {III}}, volume~60 of {\em Encyclopaedia of
  Mathematical Sciences}.
\newblock Springer-Verlag, Berlin, 1991.
\newblock Diophantine geometry.

\bibitem{laurent_exp-poly}
Michel Laurent.
\newblock \'{E}quations exponentielles-polyn\^{o}mes et suites r\'{e}currentes
  lin\'{e}aires. {II}.
\newblock {\em J. Number Theory}, 31(1):24--53, 1989.

\bibitem{Pillay}
Anand Pillay.
\newblock On groups and fields definable in {$o$}-minimal structures.
\newblock {\em J. Pure Appl. Algebra}, 53(3):239--255, 1988.

\bibitem{Pillay_GST}
Anand Pillay.
\newblock {\em Geometric stability theory}, volume~32 of {\em Oxford Logic
  Guides}.
\newblock The Clarendon Press Oxford University Press, New York, 1996.
\newblock Oxford Science Publications.

\bibitem{silverman}
Joseph~H. Silverman.
\newblock {\em The arithmetic of elliptic curves}, volume 106 of {\em Graduate
  Texts in Mathematics}.
\newblock Springer-Verlag, New York, 1992.
\newblock Corrected reprint of the 1986 original.

\bibitem{zilber2}
Boris Zilber.
\newblock Complex roots of unity on the real plane.
\newblock Available at the webpage \texttt{www.maths.ox.ac.uk/$\sim$zilber},
  2003.

\end{thebibliography}

\end{document}